\documentclass[15pt]{amsart}
\usepackage[utf8]{inputenc}
\usepackage[OT2,T1]{fontenc}
\usepackage{amscd,amssymb}
\usepackage{mathtools}
\usepackage{hyperref}
\usepackage[all]{xy}
\usepackage{microtype}
\usepackage{tikz}
\usepackage{tikz-cd}

\newcommand{\Aut}{\mathrm{Aut}}
\newcommand{\Alb}{\mathrm{Alb}}

\newcommand{\Pic}{\mathrm{Pic}}
\newcommand{\NS}{\mathrm{NS}}
\newcommand{\MP}{\mathrm{MP}}
\newcommand{\Picard}{\mathrm{Picard}}
\newcommand{\N}{\mathrm{N}}
\newcommand{\Mil}{\mathrm{Mil}}
\newcommand{\Poincar}{\mathrm{Poincar}}
\newcommand{\disc}{\mathrm{disc}}

\newcommand{\ao}{\mathbb{A}^1}

\newcommand{\G}{\Gamma}
\newcommand{\La}{\Lambda}
\newcommand{\X}{\widetilde{X_3}}
\newcommand{\C}{\mathbb{C}}
\newcommand{\Ql}{\mathbb{Q}_l}
\newcommand{\Q}{\mathbb{Q}}
\newcommand{\et}{\mathrm{H}_{\text{\'et}}}
\newcommand{\etc}{\mathrm{H}_{\text{\'et,c}}}
\newcommand{\cris}{\mathrm{H}_{\mathrm{cris}}}
\newcommand{\Omg}{\Omega}
\newcommand{\PP}{\mathbb{P}}
\newcommand{\OO}{\mathcal{O}}
\newcommand{\A}{\mathrm{Aut}}
\newcommand{\Z}{\mathbb{Z}}
\newcommand{\Sp}{\textbf{Spec}}
\newtheorem{theorem}{Theorem}[section]
\newtheorem{Assumption}[subsection]{Assumption}
\newtheorem{thm}[theorem]{Theorem}
\newtheorem{cor}[theorem]{Corollary}
\newtheorem{lemm}[theorem]{Lemma}

\newtheorem{prop}[theorem]{Proposition}

\newcommand{\Ker}{\mathrm{Ker}}
\newcommand{\chari}{\mathrm{char}}
\newcommand{\Id}{\mathrm{Id}}
\newcommand{\ddim}{\mathrm{dim}}

\newcommand{\proj}{\mathrm{proj}}

\newcommand{\Sym}{\mathrm{Sym}}
\newcommand{\Ho}{\mathrm{H}}
\newcommand{\Fa}{\mathrm{F}}
\newcommand{\Bl}{\mathrm{Bl}}
\newcommand{\DR}{\mathrm{DR}}

\newtheorem{question}[theorem]{Question}

\newtheorem{remark}[theorem]{Remark}
\theoremstyle{plain}

\title []{Automorphism and Cohomology I: Fano Varieties of Lines and Cubics}

\author {Xuanyu Pan}
\email{pan@math.wustl.edu}
\address{Department of Mathematics, Washington University in St.Louis, St.Louis, M.O. 63130}

\date{\today}
\thanks{}

\begin{document}


\begin{abstract}
In this paper, we prove a general principle of lifting automorphisms of smooth projective varieties from positive characteristic to zero characteristic. We show that the automorphism groups of the Fano varieties of smooth cubic threefolds (resp. fourfolds) act on their cohomology groups faithfully with the help of this principle. We also prove that the same result holds for Lefschetz cubic threefolds (resp. fourfolds). As an application, we classify the automorphism groups of smooth cubic fourfolds with small middle Picard numbers over complex numbers.
\end{abstract}
\maketitle
\vspace*{6pt}\tableofcontents

\section {Introduction}
A very important question in algebraic geometry is to ask whether the Torelli theorem holds for varieties. Namely, whether some algebraic structures (such as Hodge structures, F-isocrytals) on the cohomology groups of varieties can recover the geometry of varieties. Burns, Rapoport, Michael, Shafarevich, \ldots prove that the Torelli theorem holds for K3 surfaces over complex numbers $\C$ and Ogus proves that the Torelli theorem holds for supersingular K3 surfaces. The Torelli theorem says that:\\

An isomorphism $\varphi:\Ho^*(X)\simeq \Ho^*(X')$ between the cohomology groups which is preserving the algebraic structures is induced by an isomorphism $\psi : X\simeq X'$ between the varieties. 

It is natural to ask the following question:
\begin{question}
Is the map $\psi$ which satisfies $\psi^*=\varphi$ is (up to a sign) unique?
\end{question}
The question is often reduced to ask whether the automorphism group $\mathrm{Aut}(X)$ acts on the cohomology group $\Ho^*(X)$ faithfully:
\begin{equation}\label{inj}
\Aut(X)\hookrightarrow \Aut(\Ho^*(X)).
\end{equation}

For example, algebraic curves of genus at least 2 and algebraic $K3$ surfaces over an algebraically closed field have this property (\ref{inj}), see \cite{DM}, \cite{Ogus} and \cite[Chapter 15]{K3}. On the other hand, Mukai,  Namikawa and Dolgachev prove that some Enriques surfaces do not have (\ref{inj}), see \cite{Dolg} and \cite{Mukai}.

In this paper, we prove the following theorem:
\begin{theorem}\label{mainthm}
Let $X_3$ be a projective cubic threefold (resp. fourfold) with at most one ordinary double point over an algebraically closed field $k$. Suppose that $F(X_3)$ is the Fano variety of lines of $X_3$ and $l$ is a prime different from the characteristic of $k$. The natural map
\[\Aut(X_3)\hookrightarrow \Aut(\et^3(X_3,\Ql)) \text{~(resp.~} \Aut(X_3)\hookrightarrow \Aut(\et^4(X_3,\Ql))\text{)}.\]
is an injection if one of the following condition holds:
\begin{itemize}
\item $X_3$ is smooth,
\item $\chari(k)\neq 2,3$.
\end{itemize} Suppose that $Y_3$ is a smooth cubic hypersurface over $k$. The natural map
\[\Aut(\Fa(Y_3))\hookrightarrow \Aut(\et^i(\Fa(Y_3),\Ql))\]
is injective if 
\begin{itemize}
\item $i=1$, $\ddim (Y_3)= 3$ and $\chari(k)\neq 2$,
\item or if $i=2$ and $\ddim(Y_3)=4$.
\end{itemize}
\end{theorem}




We also classify the automorphism group of a smooth complex cubic fourfold whose middle Picard number is at most $2$.
\begin{theorem}
Suppose that $X_3$ is a smooth projective cubic fourfold over $\C$. The automorphism group $\Aut(X_3)$ is
\begin{itemize}
 \item trivial or $\Z/3\Z$ if the middle $\Picard$ number is one,
\item trivial if the middle $\Picard$ number is two.
\end{itemize}
\end{theorem}
As a result, one can tell whether a smooth cubic fourfold has large middle Picard number from the symmetry of its defining equation.

Let us briefly describe the structure of this paper. In Section 2, we prove that some cohomology groups are torsion-free. We provide a relation between the crystalline cohomology and the \'etale cohomology of automorphisms of a hypersurface. Meanwhile, we use the Lefschetz fixed point theorem to show that (\ref{inj}) holds for a smooth cubic threefold and its Fano surface, see Theorem \ref{ffonfanos} and Corollary \ref{corinjcubic}

In Section 3, we show that some Hodge numbers of the Fano variety of lines $\Fa(X)$ of a cubic fourfold $X$ are zeros. The result is well-known by Hodge theory over complex numbers $\C$, see \cite{Bea2}. The proof we provide here is purely algebraic. As a corollary, we show that the N\'{e}ron-Severi group of $\Fa(X)$ is torsion-free.

In Sections 4, we use a criteria for lifting automorphisms of smooth projective varieties from positive characteristic to characteristic zero. We apply the lifting criteria to reduce Theorem \ref{mainthm} in positive characteristic to characteristic zero. It follows from the holomorphic symplectic structure on $\Fa(X_3)$ for a smooth cubic fourfold $X_3$ and the techniques of projective geometry that Theorem \ref{mainthm} holds in characteristic zero. 

In Section 5, we complete the proof of Theorem \ref{mainthm}. Namely, we show that Theorem \ref{mainthm} holds for Lefschetz cubics.

In Section 6, we use the theory of Milnor fibers and vanishing cycles to show that there is a smooth cubic fourfold of middle $\Picard$ number one with an automorphism of order $3$.

 In Section 7, we use the lattice theory to classify the automorphism group of a complex cubic fourfold whose middle Picard number is small.
\\

  \textbf{Acknowledgments.} The author is very grateful for Dr.~Hang Xue and Dr.~Jie Xia for exploring cubics with him in the end of his last year in Columbia University. The author thanks Professor~Luc~Illusie for his interest in this project. The author also acknowledge Prof.~Matt Kerr , Prof.~Radu Laza, Prof.~Roya Beheshti, Prof.~Adrian Clingher and Dr.~Qi You for some useful conversations. At the end, the author thanks Prof.~Johan de Jong for giving wonderful lectures on \'etale and crystalline cohomology couple years ago. This paper is dedicated to author's friend~Dr. Zhiyu Tian who encourages the author and shares his beautiful insight about math and the world.

\section{Automorphisms of Fano Surfaces} \label{s3}

\begin{lemm}\cite[Theorem 1.1, Theorem 1.3]{PB}\label{autlin}
Suppose that $X_d$ is a smooth hypersurface of degree $d$ in $\PP^{n+1}$ over an algebraically closed field $k$. If $n\geq 1$, $d\geq 3$, and $(n, d)$
is neither $(1, 3)$ nor $(2, 4)$, then $\Aut (X_d) = \Aut_L(X_d)$ where $\Aut_L(X_d)$ is the linear automorphism group. Moreover, it is a finite group.
\end{lemm}

\begin{lemm} \label{idladic}
Let $X_d$ be a smooth hypersurface of degree $d$ in $\PP^{n+1}$ over an algebraically closed field $k$ of positive characteristic, and $f$ be a linear automorphism of $X_d$. If $n\geq 1$, $d\geq 3$, and $(n, d)$
is neither $(1, 3)$ nor $(2, 4)$, then the following are equivalent:
\begin{enumerate}
\item $f^*=\Id$ on $\et^i(X_d,\Ql)$. 
\item $f^*=\Id$ on $\cris^i(X_d/W)_K$ where $K$ is the fraction field of the Witt ring $W=W(k)$.
\end{enumerate}

\end{lemm}

\begin{proof}
By the previous lemma, we know the order of $f$ is finite. Therefore, the induced map $f^*$ on the cohomology group \[\et^i(X_d,\overline{\Ql})\text{~(resp.~}\cris^i(X_d/W)_{\overline{K}}\text{)} \]is diagonalizable. In particular, the map $f^*$ is identity on the cohomology group if and only if all the eigenvalues of $f^*$ are one. On the other hand, we have
\[\det(\Id-f^*t,\cris^i(X_d/W)_{K})=\det(\Id-f^*t,\et^i(X_d,\Ql)),\]
see \cite[Theorem 2]{KM} and \cite[3.7.3 and 3.10]{Ill1}. We have proved the lemma.
\end{proof}
The following lemma is well-known.
\begin{lemm} \label{torsionfree}
Let $X_d$ be a smooth hypersurface in $\PP^{n+1}$ over $\C$. Then, the singular cohomology $\mathrm{H}^i_{sing}(X_d,\Z)$ is torsion-free.

\end{lemm}
\begin{proof}
It is clear that we have
\[\Ho_{i}(X_d,\Z)_{\mathrm{tor}}=\Ho^{n-i}_{sing}(X_d,\Z)_{\mathrm{tor}}\]by Poincar\'e duality. Using the Lefschetz hyperplane theorem for homology and cohomology, we conclude that the cohomology group $\mathrm{H}^i_{sing}(X_d,\Z)$ is torsion-free.

\end{proof}

\begin{lemm} \label{cristorsionfree}
Let $X_d$ be a smooth hypersurface in $\PP^{n+1}$ over an algebraically closed field $k$ of positive characteristic. Then, the crystalline cohomology $\cris^i(X_d/W)$ is torsion-free.
\end{lemm}
\begin{proof}
We lift $X_d$ to a smooth hypersurface $\overline{X_d}$ over the Witt ring $W(k)=W$ of $k$. By \cite[Exp.IX Thm.1.5]{SGA7}, the Hodge-to-de Rham spectral sequence degenerates and the Hodge numbers are invariant, i.e., $h^{p,q}(X_d)=h^{p,q}((\overline{X_d})_K)$. We conclude that
\[\ddim_k\Ho^*_{\DR}(X_d/k)=\ddim_K\Ho^*_{\DR}(\overline{X_d}/K).\]
 According to the universal coefficient theorem of crystalline cohomology(\cite{Ber}), we have
\begin{equation}\label {uc}
0\rightarrow \cris^i({X_d}/W)\otimes_W k\rightarrow \Ho_{\DR}^i({X_d}/k)\rightarrow Tor_1^W(\cris^{i+1}({X_d}/W),k) \rightarrow 0.
\end{equation}
Since $\ddim_k \cris^i({X_d}/W)\otimes_W k\geq \ddim_K\cris^i({X_d}/W)_K$, we have
\[\ddim_k\Ho_{\DR}^i(X_d/k)\geq \ddim_k \cris^i({X_d}/W)\otimes k\geq  \ddim_K\cris^i({X_d}/W)_K.\]
Moreover, it follows from the comparison theorem that $\Ho^*_{\DR}(\overline{X_d}/K)=\cris^*(X_d/W)_K$. Therefore, the inequalities above are equalities.
In particular, we have \[\mathrm{Tor}_1^W(\cris^{i+1}(X/W),k)=0\]for every $i$. In other words, the cohomology group $\cris^{i+1}(X/W)$ is torsion-free.
\end{proof}

\begin{lemm}
Let $S$ be the Fano surface of a smooth cubic threefold $X_3$ over an algebraically closed field $k$. Then, we have
\[\et^2(S,\Ql)=\wedge^2 \et^1(S,\Ql)\]
\end{lemm}
\begin{proof}
Over complex numbers, it is well-known that

\begin{enumerate} 
  \item $\ddim( \et^1(S,\Ql))=10$ and
  \item $\chi(S,\Ql)=1-10+45-10+1=27$, see \cite{Compo} and \cite[Corollary 9.5, (9.14)]{CG}.
\end{enumerate}

Assume $\chari(k)$ is positive. We lift $X_3$ to a smooth cubic $X$ over the Witt ring $W(k)$ of k. We have the relative Fano scheme \[f:\Fa\rightarrow \Sp( W(k))\] of lines of $X$ over $W(k)$. It is clear that $\Fa|_{\Sp(k)}=S$. It follows from the base change theorems that the same results (i) and (ii) above hold for S.

Therefore, we have that \[\ddim_{\Q_l}(\wedge^2 \et^1(S))=\ddim_{\Q_l}(\et^2(S)).\]To prove the lemma, it suffices to show that the cup product
\[\bigcup:\bigwedge^2 \et^1(S,\Ql) \rightarrow \et^2(S,\Ql)\] is injective. It follows from \cite[Lemma 9.13]{CG} that the cup product is injective when $k$ is of characteristic zero. Assume $\chari(k)$ is positive. From the construction of the cup product for $\acute{e}$tale cohomology \cite[VI.8]{Milne}, we have the relative cup-product
\[R^1f_*(\Ql)\times R^1f_*(\Ql) \rightarrow R^2f_*(\Ql).\]
Since $f$ is smooth and proper, the higher direct image $R^if_*(\Ql)$ is a lisse l-adic sheaf for every $i$, see \cite[XVI]{SGA4} and \cite[Arcata V2,3]{SGA45}. Since $W(k)$ is strictly henselian, the sheaves $R^if_*(\Ql)$ are constant. By the proper base change theorem, to prove that the cup product $\bigcup$ is injective over $k$, it suffices to show that it is injective over $K$ where $K$ is the fraction field of $W(k)$. The injectivity of the cup product $\bigcup$ over $K$ follows from \cite[Lemma 9.13]{CG}.\\

\end{proof}
\begin{thm}\label{ffonfanos}
Suppose that $S$ is the Fano surface of lines of a smooth cubic threefold over an algebraically closed field $k$. Let $l$ be a prime different from the characteristic of $k$. The natural map
\[j:\Aut(S)\rightarrow \Aut(\et^1(S,\Ql))\] is injective if $\chari(k)\neq 2$.
\end{thm}
\begin{proof}
First, we claim an automorphism $g$ of $S$ has a fixed point if it acts trivially on the first l-adic cohomology group. From the previous lemma and Poincar$\acute{e}$ duality, it follows that $g$ acts trivially on $\et^i(S,\Ql)$ for all $i$. Suppose that $g$ has no any fixed point. By the Lefschetz fixed-point theorem, we have a contradiction as follows
\[0=\#(\mathrm{Fix}(g))=\sum\limits_{i=0}^4 (-1)^i Tr(g^*,\et^i(S,\Ql))=1-10+45-10+1=27.\]
Suppose that $s_0 \in S$ is a fixed point of $g:S\rightarrow S$. Since $\chari(k)\neq 2$, the natural morphism \[\lambda:S \rightarrow \Alb(S)\] is an immersion where $\Alb(S)$ is the Albanese variety of $S$, see \cite{Bea}. Moreover, we have \[\et^1(S,\Ql)=\et^1(\Alb(S),\Ql)\]via the pullback $\lambda^*$. Assume $\lambda(s_0)=0$. By the universal property of $\Alb(S)$, we have a morphism \[h:\Alb(S) \rightarrow \Alb(S)\] such that $\lambda\circ g=h\circ \lambda$ and $h(\lambda(s_0))=\lambda(s_0)$.
Therefore, the morphism $h$ is an automorphism of the group scheme $\Alb(S)$(\cite[Chapter 2]{mum}). Obviously, we have $h^*=\Id$ on $\et^1(\Alb(S),\Ql)$. By a theorem of Tate (\cite[Chapter IV]{mum}), we know $h=\Id_{\Alb(S)}$. It implies that $g=\Id_S$.
\end{proof}

\begin{lemm} \label{re}
With the assumptions and notations as above, we have a natural diagram as follows:
\begin{equation}\label{abjcoh}
\xymatrix{\Aut(X_3) \ar@{}[rd]|-{\square} \ar[r]^j \ar[d] & \Aut(S)\ar[d]\\
\Aut(\et^3(X_3,\Ql)) \ar@{=}[r] & \Aut(\et^1(S,\Ql))}
\end{equation}
where the identification at the bottom of the diagram is the Abel-Jacobi map, see \cite{CG}. 
Moreover, the morphism $j$ is injective.

\end{lemm}
\begin{proof}
Let $g$ be an automorphism of $X_3$. By Lemma \ref{autlin}, the automorphism $g$ of $X_3$ is linear. Therefore, the automorphism $g$ maps lines to lines. It gives rise to a natural morphism $j:\Aut(X_3)\rightarrow \Aut(S)$. 

The proof of the commutativity of the diagram is similar to the proof of Lemma \ref{fre}, so we postpone the proof to Section 5. On the other hand, for a point $p\in X_3$, there are at least two lines $l_1$ and $l_2$ passing through $p$ (i.e., $p=l_1 \cap l_2$). If $j(g)$ is the identity, then we have $$g(p)=j(g)(l_1)\cap j(g)(l_2)=l_1\cap l_2=p.$$It follows that the map $g$ is the identity.
\end{proof}

\begin{cor} \label{corinjcubic}
Let $X_3$ be a smooth cubic threefold over an algebraically closed field $k$. Suppose that $l$ is a prime different from the characteristic of $k$. Then the automorphism group $\Aut(X_3)$ acts on $\et^3(X_3,\Ql)$ faithfully.
\end{cor}
\begin{proof}
Assume $\chari(k)\neq 2$. The corollary follows from the previous lemma and Theorem \ref{ffonfanos}. For $\chari(k)=2$, we refer to \cite{PAN}.
\end{proof}
\section {The First de Rham Cohomology Groups}
Our main result of this section is Proposition \ref{nsgroup}. The proof is tedious. We suggest the reader to skip it for the first reading.

Suppose that $X_3$ is a smooth cubic fourfold in $\PP^{5}$ over an algebraically closed field k. Denote by $F(X_3)$ the Fano variety of lines of $X_3$. We use the notations following \cite{Compo}. 

Let $V$ be a vector space of dimension $6$ over $k$. Denote by $P$ the projective space $\PP(V)=\PP^5$. Let $G$ be the Grassmannian $\textbf{Grass}_2(V)$. The fundamental sequence on $G$ is:
\begin{equation} \label{fun}
0\rightarrow M\rightarrow V_G\rightarrow Q\rightarrow 0
\end{equation}
where $M$ is the universal subbundle and $Q$ is the quotient bundle. Denote by $\G$ the incidence correspondence subscheme of $P\times G$, see \cite[Page 36]{Compo}. In particular, we have \[\G=\{(x,l)|x\in l\}\] and two maps $p$ and $q$ as follows
\[\xymatrix{G &\PP(Q)=\G=\PP(K) \ar[l]_>>>>{q} \ar[r]^<<<{p}& P}\]
where $K=\Omega_{P}^1(1)$. Note that there is a short exact sequence \[0\rightarrow K\rightarrow V_P\rightarrow \OO_P(1)\rightarrow 0\]where $V_P=\OO_P^{\oplus 6}$. Let R (resp. N) be the fundamental sheaf $\OO_{\PP(Q)}(1)=p^*\OO_P(1)$ (resp. $\OO_{\PP(K)}(1)$). We denote $q^*Q$ by $Q_{\G}$.
We summarize some results from \cite{Compo} in the following proposition.
\begin{prop}\cite[5.2.2-5.4.2]{Compo} \label{comprop}\\
We have a short exact sequence \[0\rightarrow N\rightarrow Q_{\G}\rightarrow R\rightarrow 0.\] Therefore, we have $\wedge^2 Q_{\G}=N \otimes R$. Moreover, the dualizing sheaf $\omega_{\G/k}$ is $N^{-5}\otimes R^{-7}$. We have \[\wedge ^j E(n)=\wedge^{(4-j)}(\Sym_3(Q))\otimes (\wedge^2 Q)^{\otimes (n-6)}\]where $E=\Sym_3(Q)^{\vee}$.
\end{prop}
\begin{lemm} \label{vanish}
Some cohomology groups are zeros as follows.
\begin{enumerate}
\item $\Ho^i(\G,N^s\otimes R^t)=0$ if $-1\geq s \geq -4$, or if $s=0$ and $i\notin \{0,5\}$, or if $s=0$, $i=0$ and $t\leq -1$, or if $s=1$ and $i\neq 1$, or if $s=1, i=1$ and $t\neq -1$, or if $s=-5$ and $i\notin \{4,9\}$, or if $s=-6$ and $i\notin \{4,8,9\}$.
\item $\Ho^i(\G,Q_{\G}\otimes N^s\otimes R^t)=0$ if $-2\geq s \geq -4$, or if $s=-5$ and $i \notin \{4,9\}$.
\item $\Ho^i(\G, \Sym_2(Q_{\G})\otimes N^s\otimes R^t)=0$ if $s=-3$ or $s=-4$, or if $s=-5$, $i\notin \{9,4\}$, or if $s=-5$, $i=4$ and $t\leq -4$, or if $s=-5$, $i=9$ and $t\geq -8$, or if $s=-6$ and $i\notin \{8,9,4\}$.
\item $\Ho^i(\G, \Sym_3(Q_{\G})\otimes N^s\otimes R^t)=0$ if $s=-5$ and $i\notin \{4,9\}$, or if $s=-5$, $i=4$ and $t\leq -5$, or if $s\in \{-6,-7\}$ and $i\in \{2,3\}$.
\item $\Ho^i(\G, \wedge^2\Sym_2(Q_{\G})\otimes N^s\otimes R^t)=0$ if $s\in \{-4,-5\}$, or if $s=-3$ and $i\notin \{0,5\}$, or if $s=-3$, $i=-5$ and $t\geq -6$.
\item $\Ho^i(\G, \wedge^2 \Sym_3(Q_{\G})\otimes N^s\otimes R^t)=0$ if $s=-5$ and $i\notin \{0,5\}$, or if $s=-5$, $i=5$ and $t\geq -6$, or if $s=-6, i\notin \{4,9\}$, or if $s=-6$, $i=4$ and $t\leq -5$, or if $s=-6$, $i=9$ and $t\geq -9$ or if $s=-7$ and $i\in \{1,2\}$.
\end{enumerate}
\end{lemm}
\begin{proof}
The lemma follows from Proposition \ref{comprop}, the formula \cite[5.7.1, 5.7.2]{Compo}, Serre duality and the Bott vanishing theorem. The proof is straightforward but tedious.

\begin{enumerate}
\item By \cite[5.7.1]{Compo}, the cohomology groups vanish if $-1\geq s \geq -4$. If $s=0$ and $i\notin\{0,5\}$ or if $s=i=0$ and $t\leq -1$, then, by \cite[5.7.1]{Compo}, we have $\Ho^i(\G,N^s\otimes R^t)=\Ho^i(P,\OO_P(t))=0$. If $s=1$ then, by \cite[5.7.1,5.7.2]{Compo}, we have that $\Ho^i(\G,N^s\otimes R^t)=\Ho^i(P,\Omega^1_P(t+1))=0$ for $i\neq 1$ or \[i=1 \text{~but~} t\neq-1\]where we use the Bott vanishing theorem (see \cite[Bott's Formula]{Bott}). \\
For $s=-5$ and $i\notin \{4,9\}$, we conclude $\Ho^i(N^{-5}\otimes R^{t})=0$ by Serre dualiy and \cite[5.7.1]{Compo}. In fact, we have\[\Ho^i(N^{-5}\otimes R^{t})=\Ho^{9-i}(\G, \omega_{\G} \otimes N^5R^{-t})=\Ho^{9-i}(P,\OO_P(-t-7))=0.\]
For $s=-6$ and $i\notin \{4,8,9\}$, we conclude $\Ho^i(N^{s}\otimes R^{t})=0$ by Serre dualiy and \cite[5.7.1]{Compo}. In fact, we have \[\Ho^i(N^{-6}\otimes R^{t})=\Ho^{9-i}(\G, \omega_{\G} \otimes N^6R^{-t})=\Ho^{9-i}(P,\Omega^1_P(-t-6))=0\]by \cite[5.7.1]{Compo} and the Bott vanishing theorem. \\
In summary, we have proved (i).\\

\item Applying the short exact sequence in Proposition \ref{comprop} and \cite[5.7.1]{Compo}, 
we have a long exact sequence as follows\[\ldots \rightarrow (0=)\Ho^i(\G,N^{s+1}\otimes R^{t}) \rightarrow \Ho^i(\G,Q_{\G}\otimes N^{s}\otimes R^{t})\rightarrow \Ho^i(\G,N^{s}\otimes R^{t+1})(=0)\rightarrow \ldots \] if $-2\geq s \geq -4$. \\
For $s=-5$ and $i\notin \{9,4\}$, it follows from (i) that $\Ho^i(\G,N^{s+1}\otimes R^{t})=0$ and $\Ho^i(\G,N^{s}\otimes R^{t+1})=0$. In summary, we have proved (ii).\\

\item By \cite[5.11.5]{Compo}, we have a short exact sequence as follows
\begin{equation}\label{sexsq}
0\rightarrow Q_{\G}\otimes N^{s+1}\otimes R^{t} \rightarrow \Sym_2(Q_{\G})\otimes N^s\otimes R^t \rightarrow N^s\otimes R^{t+2}\rightarrow 0.
\end{equation}
Using (ii), we have $ \Ho^i(\G,\Sym_2(Q_{\G})\otimes N^{s}\otimes R^{t})=0$ if $s=-3$ or $-4$.\\

For $s=-5$ and $i\notin \{4,9\}$, we have the long exact sequence associated to (\ref{sexsq}) as follows:
\[\ldots \rightarrow \Ho^i(Q_{\G}\otimes N^{s+1}R^t )\rightarrow \Ho^i(\Sym_2(Q_{\G})\otimes N^{s} R^{t}) \rightarrow \Ho^i(N^{s}R^{t+2}) \rightarrow \ldots\]
where  $\Ho^i(N^{-5}R^{t+2})=0$ and $\Ho^i(Q_{\G}\otimes N^{-4}R^t )=0$ by (i) and (ii).\\
For $s=-5$, $i=4$ and $t\leq -4$, we have $\Ho^i(Q_{\G}\otimes N^{-4}R^t )=0$ by (ii) and \[\Ho^i(N^{-5}R^{t+2})=\Ho^{9-i}(R^{-t-9})=\Ho^5(P,\OO_P(-t-9))=0\]by \cite[5.7.1]{Compo}.\\
For $s=-5$, $i=9$ and $t\geq -8$, we know that $\Ho^i(Q_{\G}\otimes N^{-4}R^t )=0$ and \[\Ho^i(N^{-5}R^{t+2})=\Ho^{9-i}(R^{-t-9})=\Ho^0(P,\OO_P(-t-9))=0\]by \cite[5.7.1]{Compo}.\\
Similarly, for $s=-6$ and $i\notin \{4,8,9\}$, it follows from (i) and (ii) that $\Ho^i(Q_{\G}\otimes N^{-5}R^t )=0$ and $\Ho^i(\G,N^{-6}R^{t+2})=0$. We show (iii).
\item By \cite[Page 42]{Compo}, we have a short exact sequence
\begin{equation}\label{eq3}
0\rightarrow \Sym_2(Q_{\G})\otimes N\rightarrow \Sym_3(Q_{\G})\rightarrow R^3 \rightarrow 0.
\end{equation}
It gives rise to a long exact sequence after tensoring $N^sR^t$ as follows.
\[\ldots\rightarrow \Ho^i(\Sym_2(Q_{\G})\otimes N^{s+1}R^t) \rightarrow \Ho^i(\Sym_3(Q_{\G})\otimes N^s R^t) \rightarrow \Ho^i(N^s\otimes R^{t+3})\rightarrow \ldots\]
If $s=-5$ and $i\notin \{4,9\}$, then we have \[\Ho^i(N^s\otimes R^{t+3})=0 \text{~and~} \Ho^i(\Sym_2(Q_{\G})\otimes N^{s+1}R^t)=0\]by (i) and (iii).
If $s=-5$, $i=4$ and $t\leq -6$, then we have \[\Ho^i(\Sym_2(Q_{\G})\otimes N^{s+1}R^t)=0\] by (iii) and the fact \[\Ho^i(N^s\otimes R^{t+3})=\Ho^{9-i}(\G, R^{-t-10})\]\[=\Ho^5(P,\OO_P(-t-10))=\Ho^0(P,\OO_P(t+4))=0\] where we use Serre duality and \cite[5.7.1]{Compo}.\\
For $s\in\{-6, -7\}$ and $i\in \{2,3\}$, the exact sequence (\ref{eq3}) gives rise to an exact sequence
\[\Ho^i(\Sym_2(Q_{\G})\otimes N^{s+1}\otimes R^t) \rightarrow \Ho^i(\Sym_3(Q_{\G})\otimes N^{s}\otimes R^t) \rightarrow \Ho^i(N^s\otimes R^{t+3})\] where the first term is zero by (iii). We claim the last term is also zero. In fact, it follows from (i) for $s=-6$. For $s=-7$, we have \[\Ho^i(\G, N^s\otimes R^{t+3})=\Ho^{9-i}(\G, N^{-s-5}\otimes R^{-t-10})\]\[=\Ho^{9-i}(\PP^5, \Sym_2(K)(-t-9))=0\]by Serre duality and \cite[5.7.1]{Compo}. Therefore, we have proved (iv).

\item According to \cite[Lemma 2.10, 5.11.5]{Compo}, we have a short exact sequence (\cite[Page 42]{Compo})
\[0\rightarrow N^{s+3} \otimes R^{t+1}\rightarrow \wedge^2\Sym_2(Q_{\G})\otimes N^s\otimes R^t\rightarrow Q_{\G}\otimes N^{s+1}\otimes R^{t+2}\rightarrow 0.\]It gives rise to a long exact sequence
\[\ldots\rightarrow \Ho^i(N^{s+3}R^{t+1})\rightarrow \Ho^i(\wedge^2 \Sym_2(Q_{\G})\otimes N^s R^t)\rightarrow \Ho^i( Q_{\G}\otimes N^{s+1} R^{t+2}) \rightarrow \ldots\] We notice that if $s=-3$, $i=5$ and $t\geq -6$, then $\Ho^i(N^{s+3}\otimes R^{t+1})=\Ho^i(P,\OO_P(t+1))=0$ by \cite[5.7.1]{Compo}. By (i) and (ii), we conclude that \[ \Ho^i(\G, N^{s+3}R^{t+1})=\Ho^i( \G,Q_{\G}\otimes N^{s+1} R^{t+2}) =0\]under the hypotheses of (v). Therefore, we have that \[\Ho^i(\G, \wedge^2 \Sym_2(Q_{\G})\otimes N^s R^t)=0.\]We have proved (v).
\item As in \cite[Page 42]{Compo}, we have the following short exact sequence\[0\rightarrow \wedge^2 \Sym_2(Q_{\G})\otimes N^{s+2}\otimes R^t \rightarrow \wedge^2 \Sym_3(Q_{\G})\otimes N^s \otimes R^t \rightarrow \Sym_2(Q_{\G})\otimes N^{s+1}\otimes R^{t+1} \rightarrow 0.\] which gives rise to a long exact sequence
\[\ldots \rightarrow \Ho^i(\wedge^2 \Sym_2 \otimes N^{s+2}R^t) \rightarrow  \Ho^i(\wedge^2 \Sym_3\otimes N^s R^t)\rightarrow  \Ho^i(\Sym_2\otimes N^{s+1}R^{t+1})\rightarrow \ldots\]
By (iii) and (v), we conclude that
\begin{center}
 $\Ho^i(\G,\wedge^2 \Sym_2(Q_{\G})\otimes N^{s+2}\otimes R^t )$ and $\Ho^i(\G, \Sym_2(Q_{\G})\otimes N^{s+1}\otimes R^{t+1})$
\end{center}are zeros under the hypotheses of (vi). We have proved (vi).
\end{enumerate}
\end{proof}

\begin{prop}\label{nsgroup}With the notations as above, we have
\[\Ho^1(\Fa(X_3),\OO_{\Fa(X_3)})=\Ho^0(\Fa(X_3),\Omg^1_{\Fa(X_3)})=\Ho^3(\Fa(X_3),\OO_{\Fa(X_3)})=0.\]Therefore, the N\'eron-Severi group $\NS(\Fa(X_3))$ of $\Fa(X_3)$ is torsion free. Moreover, we have that $\Ho^2(\Fa(X_3),\OO_{\Fa(X_3)})=k.$
\end{prop}
\begin{proof}
First of all, we show that $\Ho^1(\Fa(X_3),\OO_{\Fa(X_3)})=\Ho^3(\Fa(X_3),\OO_{\Fa(X_3)})=0$. In fact, the Fano variety $\Fa(X_3)$ of $X_3$ is defined by a regular section $s$ of the vector bundle $\Sym_3(Q)=E^{\vee}$ which is of rank $4$, see \cite[Theorem 1.3]{Compo}. So we have the following exact sequence
\begin{equation}\label{ex}
\xymatrix{0\ar[r] &\wedge ^4 E \ar[r] &\wedge^3 E\ar[r] &\wedge^2 E\ar[r] & E\ar[r]^{s} & \OO_{G}\ar[r] & \OO_{F}\ar[r] & 0.}
\end{equation}

By the syzygy argument, in order to prove that $\Ho^1(F,\OO_F)=0$, it suffices to prove that \[\Ho^1(G,\OO_G)=\Ho^2(G,E)=\Ho^3(G,\wedge^2 E)=\Ho^4(G,\wedge^3 E)=\Ho^5(G, \wedge^4 E)=0.\]
This is verified by \cite[Theorem 5.1]{Compo}. Similarly, to prove that $\Ho^3(\Fa,\OO_{\Fa})=0$, it suffices to prove that
\[\Ho^3(G,\OO_G)=\Ho^4(G,E)=\Ho^5(G,\wedge^2 E)=\Ho^6(G,\wedge^3 E)=\Ho^7(G, \wedge^4 E)=0\]
which is also verified by \cite[Theorem 5.1]{Compo}. To prove that $\Ho^2(F,\OO_F)=k$, it suffices to prove that
\[\Ho^2(\OO_G)=\Ho^3(\OO_G)=\Ho^2(E)\]\[=\Ho^3(E)=\Ho^4(\wedge^3 E)=\Ho^5(\wedge^3 E)\]\[=\Ho^5(\wedge^4 E)=\Ho^6(\wedge^4 E)=0\] and $\Ho^4(G,\wedge ^2 E)=k$ which are verified by \cite[Theorem 5.1]{Compo}.
\\
Apply Lemma \ref{v1} and Lemma \ref{v2} to the long exact sequence of cohomology groups associated to the short exact sequence
\[0\rightarrow E|_{\Fa(X_3)}\rightarrow \Omega_{G}^1|_{\Fa(X_3)} \rightarrow \Omega_{{\Fa(X_3)}}^1\rightarrow 0.\]Note that $E|_{\Fa(X_3)}$ is the conormal bundle of ${\Fa(X_3)}$ in $G$. We conclude that $\Ho^0(\Fa(X_3),\Omg^1_{\Fa(X_3)})=0$.

In summary, we conclude $\Ho^1_{\DR}(\Fa(X_3))=0$ since the Hodge numbers\[\Ho^1(\Fa(X_3),\OO_{\Fa(X_3)})=\Ho^0(\Fa(X_3),\Omg^1_{\Fa(X_3)})\]are zeros. 

By the universal coefficient theorem of crystalline cohomology, we know $\cris^2(\Fa(X_3)/W)$ is torsion-free. By a theorem of Illusie and Deligne, see \cite[Remark 3.5]{delsur} and \cite{Ill}, we have an injection \[\NS(\Fa(X_3))\otimes \Z_p \hookrightarrow \cris^2(\Fa(X_3)/W).\] We conclude that $\NS(\Fa(X_3))$ is $p$-torsion-free. On the other hand, we have the short exact sequence \cite{Milne}
\[0\rightarrow \NS(\Fa(X_3))\otimes \Z_l\rightarrow \et^2(\Fa(X_3),\Z_l(1))\rightarrow T_l(\mathrm{Br}(\Fa(X_3)))\rightarrow 0.\] We claim that $\et^2(\Fa(X_3),\Z_l(1))$ is torsion free. Therefore, the group $\NS(\Fa(X_3))$ is torsion-free.

In fact, the variety $\Fa(X_3)_{\mathbb{C}}$ is simply connected. By the universal coefficient theorem, we have
\[\Ho^2_{sing}(\Fa(X_3)_{\mathbb{C}},\Z_l)=Hom(\Ho_2(\Fa(X_3)_{\mathbb{C}},\Z),\Z_l)=\lim_{\leftarrow n } Hom(\Ho_2(\Fa(X_3)_{\mathbb{C}},\Z), \Z/l^n\Z)\]
\[=\lim_{\leftarrow n} \et^2 (\Fa(X_3)_{\mathbb{C}},\Z/l^n\Z)=\et^2(\Fa(X_3)_{\mathbb{C}},\Z_l).\]
Since $\Ho^2_{sing}(\Fa(X_3)_{\mathbb{C}},\Z)$ is torsion free, the group $\Ho^2_{sing}(\Fa(X_3)_{\mathbb{C}},\Z_l)$ is torsion-free. We show the claim holds in characteristic zero. Assume $\mathrm{char}(k)>0$. Since $\Fa(X_3)$ over $k$ has a lifting from positive characteristic to characteristic zero, the claim holds for $\Fa(X_3)$ over $k$.
\end{proof}
\begin{lemm} \label{v1}
With the notations as above, we have $\Ho^1(F,E|_F)=0$.
\end{lemm}
\begin{proof}
By the exact sequence (\ref{ex}) and the syzygy arugment, it suffices to prove that
\[\Ho^5(G,\wedge^4 E\otimes E)=\Ho^4(\wedge^3 E\otimes E)=\Ho^3 (\wedge^2 E\otimes E)=\Ho^2(E\otimes E)=\Ho^1(G, E)=0.\]
\begin{enumerate}
\item We claim $\Ho^5(G,\wedge^4 E\otimes E)=0$. In fact, we have $\wedge ^4 E=(\wedge ^2 Q)^{-6}$ by Proposition \ref{comprop}. It follows from \cite[a formula for E Page 43]{Compo} that  \[E=R^1q_*(R^{-5})\otimes \wedge^2 Q.\] By the fact $R^0q_*(R^{-5})=0$ and Proposition \ref{comprop}, we conclude that $\wedge^4 E\otimes E$ is\[R^1q_*(R^{-5})\otimes (\wedge^2 Q)^{-5}=Rq_*(R^{-5}\otimes q^*(\wedge^2 Q)^{-5})=Rq_*(R^{-10}\otimes N^{-5}).\]Therefore, we show that
\[\Ho^5(G,\wedge^4 E\otimes E)=\Ho^5(\G, R^{-10}\otimes N^{-5})=\Ho^4(\G, R^3)=\Ho^4(P,\OO_P(3))=0.\]
\item We claim $\Ho^4(\wedge^3 E\otimes E)=0$. In fact, by Proposition \ref{comprop} and \cite[a formula of E in Page 43]{Compo}, we conclude that $\wedge^3 E\otimes E$ is
\[ \Sym_3(Q)\otimes (\wedge^2 Q)^{-5} \otimes R^1q_*(R^{-5})=\Sym_3(Q)\otimes (\wedge^2 Q)^{-5} \otimes Rq_*(R^{-5})\]
\[=Rq_*(\Sym_3(Q_{\G})\otimes N^{-5}R^{-10}).\]
Therefore, by Lemma \ref{vanish} (iv), we conclude that \[\Ho^4(G,\wedge^3 E\otimes E)=\Ho^4(\G,\Sym_3(Q_{\G})\otimes N^{-5}R^{-10})=0.\]

\item We claim that $\Ho^3(G,\wedge^2 E\otimes E)=0$. In fact, as before
\[\wedge^2 E\otimes E=\wedge^2 \Sym_3(Q)\otimes (\wedge ^2 Q)^{-5}\otimes Rq_*(R^{-5})\]\[=Rq_*(\wedge^2 \Sym_3(Q_{\G})\otimes N^{-5}\otimes R^{-10}).\]It follows that \[\Ho^3(G,\wedge^2 E\otimes E)=\Ho^3(\G,\wedge^2 \Sym_3(Q_{\G})\otimes N^{-5}\otimes R^{-10})=0.\]
The last equality $\Ho^3(\G,\wedge^2 \Sym_3(Q_{\G})\otimes N^{-5}\otimes R^{-10})=0$ follows from Lemma \ref{vanish} (vi).

\item We claim that $\Ho^2(G,E\otimes E)=0$. In fact, as before, we have that
\[E\otimes E=Rq_*(R^{-5}\otimes \wedge^2 Q_{\G}\otimes q^*E)\]\[=Rq_*(R^{-4}\otimes N \otimes q^*E)=Rq_*(R^{-4}\otimes N \otimes \Sym_3(Q_{\G})^{\vee})\]
where the last equality follows from $q^* E=\Sym_3(Q_{\G})^{\vee}$. Therefore, we have that $$\Ho^2(G,E\otimes E)=\Ho^2(\G, \Sym_3(Q_{\G})^{\vee}\otimes N\otimes R^{-4}).$$
The short exact sequence (\ref{eq3}) gives rise to a short exact sequence as follows
\[0 \rightarrow N\otimes R^{-7} \rightarrow \Sym_3(Q_{\G})^{\vee}\otimes N\otimes R^{-4}\rightarrow \Sym_2(Q_{\G})^{\vee}\otimes R^{-4}\rightarrow 0.\]By Lemma \ref{vanish} (i), we have $\Ho^2(\G,N\otimes R^{-7})=0$. To prove the claim, it suffices to prove that $\Ho^2(\G, \Sym_2(Q_{\G})^{\vee}\otimes R^{-4})=0$.  In fact, by the exact sequence \cite[5.11.5]{Compo}, we have an exact sequence as follows
\[ 0\rightarrow R^{-6}\rightarrow \Sym_2(Q_{\G})^{\vee}\otimes R^{-4}\rightarrow Q_{\G}^{\vee}\otimes N^{-1}\otimes R^{-4}\rightarrow 0.\]
By Lemma \ref{vanish} (i), we have $\Ho^2(\G,R^{-6})=0$. 

In the following, we show that $\Ho^2(\G, Q_{\G}^{\vee}\otimes N^{-1}\otimes R^{-4})=0$. In fact, by Proposition \ref{comprop}, we have a short exact sequence as follows
\[0\rightarrow N^{-1}\otimes R^{-5}\rightarrow Q_{\G}\otimes N^{-1}\otimes R^{-4} \rightarrow N^{-2}\otimes R^{-4} \rightarrow 0.\]
We conclude that  $\Ho^2(\G, Q_{\G}^{\vee}\otimes N^{-1}\otimes R^{-4})=0$ since \[\Ho^2(N^{-1}\otimes R^{-5})=\Ho^2(N^{-2}\otimes R^{-4})=0\]by Lemma \ref{vanish} (i). Therefore, we show that $\Ho^2(\G, \Sym_2(Q_{\G})^{\vee}\otimes R^{-4})=0$.
\item $\Ho^1(G,E)=0$ follows from \cite[Theorem 5.1]{Compo}. 
\end{enumerate}

\end{proof}
\begin{lemm}\label{pullback}
For any vector bundle $V$ on $G$, we have $\Ho^i(\G,q^*V)=\Ho^i(G,V)$.
\end{lemm}
\begin{proof}
The lemma follows from the projection formula and $R^iq_*\OO_{\G}=0$ for $i>0$. In fact, we have
\[\Ho^i(G,V)=\Ho^i(G,V\otimes q_*\OO_G)=\Ho^i(G,q_*q^*V)=\Ho^i(G,Rq_*q^*V)=\Ho^i(\G,q^*V)\]where the second equality follows from $q_*\OO_{\G}=\OO_{G}$.
\end{proof}

\begin{lemm}\label {v2}
With the notations as above, we have that
$\Ho^0(F,\Omega^1_G|_F)=0$.

\end{lemm}
\begin{proof}
By the exact sequence (\ref{ex}), we have an exact sequence as follows
\[\xymatrix{\wedge^4E \otimes \Omega_G\ar@{^{(}->}[r] & \wedge^3E\otimes \Omega_G\ar[r] & \wedge^2E\otimes \Omega_G\ar[r] & E\otimes \Omega_G\ar[r] & \Omega_G\ar@{->>}[r] &\Omega_G|_F}.\]
By the syzygy arugment, to prove the lemma, it suffices to show that
\[\Ho^4(\wedge^4E\otimes \Omega_G)=\Ho^3(\wedge^3E\otimes \Omega_G)=\Ho^2(\wedge^2 E\otimes \Omega_G)=\Ho^1(E\otimes \Omega_G)=\Ho^0(G, \Omega_G)=0.\] We first notice that $\Omega_G^1=Q^{\vee}\otimes M$.
\begin{enumerate}
\item By the exact sequence (\ref{fun}), we have that \[0\rightarrow Q^{\vee}\otimes M \rightarrow Q^{\vee} \rightarrow Q^{\vee}\otimes Q\rightarrow 0.\] We claim $\Ho^0(G,Q^{\vee})=0$. Therefore, we have $\Ho^0(G,\Omega_G^{1})=\Ho^0(G,Q^{\vee}\otimes M)=0$. In fact, by the exact sequence in Proposition \ref{comprop}, we have a short exact sequence \[0\rightarrow N^{-1}\rightarrow Q_{\G}^{\vee}\rightarrow R^{-1}\rightarrow 0.\]It follows from Lemma \ref{vanish} (i) that \[\Ho^0(\G,R^{-1})=\Ho^{0}(\G, N^{-1})=0.\]We have proved the claim.

\item We claim that $\Ho^1(G,E\otimes \Omega^1_{G})=0$. In fact,  as in \cite[Page 43]{Compo}, we have $E=R^1q_*(R^{-5})\otimes \wedge^2 Q$. Denote by $M_{\G}$ the pullback of $M$ via $q$. We conclude that \[\Ho^1(E\otimes \Omega_G)=\Ho^1(G,Rq_*(R^{-5}\otimes \wedge^2 Q_{\G}\otimes Q_{\G}^{\vee}\otimes M_{\G}))\]\[=\Ho^1(\G,R^{-5}\otimes Q_{\G}\otimes M_{\G})\] where the last equality follows from $Q_{\G}=\wedge^2Q_{\G}\otimes Q_{\G}^{\vee}$. By the short exact sequence\[ 0\rightarrow M_{\G}\otimes Q_{\G}\otimes R^{-5} \rightarrow Q_{\G}^{\oplus 6}\otimes R^{-5} \rightarrow Q_{\G}\otimes Q_{\G}\otimes R^{-5}\rightarrow 0,\]
to prove the claim, it suffices to prove that \[\Ho^0(\G,Q_{\G}\otimes Q_{\G}\otimes R^{-5})=\Ho^1(\G,Q_{\G}\otimes R^{-5})=0.\]In fact, it follows from Lemma \ref{v} (i) below.

\item We claim that $\Ho^2(G, \wedge^2 E\otimes \Omega_{G})=0$. In fact, by \cite[5.4.2]{Compo} and Lemma \ref{pullback}, we have
\[\Ho^2(G,\wedge^2 E\otimes \Omega_{G})=\Ho^2(\G,\wedge^2\Sym_3(Q_{\G})\otimes (\wedge^2 Q_{\G})^{-6}\otimes Q_{\G}^{\vee}\otimes M_{\G}).\]From the short exact sequence
\begin{equation} \label{eq1}
0\rightarrow M_{\G}\rightarrow V_{\G}\rightarrow Q_{\G}\rightarrow 0,
\end{equation}
 we have the exact cohomology sequence as follows
\[\Ho^1(\wedge^2\Sym_3(Q_{\G})\otimes (\wedge^2 Q_{\G})^{-6}\otimes Q_{\G}^{\vee}\otimes Q_{\G})\rightarrow \Ho^2(\wedge^2\Sym_3(Q_{\G})\otimes (\wedge^2 Q_{\G})^{-6}\otimes Q_{\G}^{\vee}\otimes M_{\G})\] \[\rightarrow \Ho^2(\wedge^2\Sym_3(Q_{\G})\otimes (\wedge^2 Q_{\G})^{-6}\otimes Q_{\G}^{\vee})^{\oplus 6}.\]By Lemma \ref{v} (ii) and (iii) below, we know the first and last term of this cohomology sequence are zeros. Therefore, we prove the claim.

\item We claim $\Ho^3(\wedge^3E\otimes \Omega_G)=0$. In fact, by \cite[5.4.2]{Compo}, Proposition \ref{comprop} and Lemma \ref{pullback}, we have
\[\Ho^3(G,\wedge^3E\otimes \Omega_G)=\Ho^3(\G, \Sym_3(Q_{\G})\otimes N^{-6}\otimes R^{-6}\otimes Q_{\G}^{\vee}\otimes M_{\G}).\] From the short exact sequence (\ref{fun}), we have an exact cohomology sequence
\[\Ho^2(\Sym_3(Q_{\G})\otimes N^{-6}\otimes R^{-6}\otimes Q_{\G}^{\vee}\otimes Q_{\G})\rightarrow \Ho^3(\G, \Sym_3(Q_{\G})\otimes N^{-6}\otimes R^{-6}\otimes Q_{\G}^{\vee}\otimes M_{\G})\] \[\rightarrow \Ho^3(\Sym_3(Q_{\G})\otimes N^{-6}\otimes R^{-6}\otimes Q_{\G}^{\vee})^{\oplus 6}.\] Since the first term and last term of this sequence are zeros by Lemma \ref{v} (iv) and (v), we prove the claim.

\item We claim $\Ho^4(\wedge^4E\otimes \Omega_G)=0$. In fact, by \cite[5.4.2]{Compo} and Lemma \ref{pullback}, we have
\[\Ho^4(\wedge^4E\otimes \Omega_G)=\Ho^4(\G,N^{-6}\otimes R^{-6}\otimes Q_{\G}^{\vee} \otimes M_{\G}).\]The exact sequence (\ref{fun}) gives rise to an exact cohomology sequence as follows
\[\Ho^3(N^{-6}\otimes R^{-6} \otimes Q_{\G}\otimes Q_{\G}^{\vee})\rightarrow \Ho^4(N^{-6}\otimes R^{-6} \otimes M_{\G}\otimes Q_{\G}^{\vee}) \rightarrow \Ho^4(Q_{\G}\otimes N^{-6}\otimes R^{-6})^{\oplus 6}\]where the first and last terms are zeros by Lemma \ref{v} (vii) (viii) below. Therefore, we prove the claim.

In summary, we prove the Lemma.

\end{enumerate}

\end{proof}

\begin{lemm}\label{v}
With the notations as before, we have that
\begin{enumerate}
\item $\Ho^0(\G, Q_{\G}\otimes Q_{\G}\otimes R^{-5})=\Ho^1(\G,Q_{\G}\otimes R^{-5})=0$
\item $\Ho^2(\G, \wedge^2\Sym_3(Q_{\G})\otimes (\wedge^2 Q_{\G})^{-6}\otimes Q_{\G}^{\vee})=0$
\item $\Ho^1(\G, \wedge^2\Sym_3(Q_{\G})\otimes (\wedge^2 Q_{\G})^{-6}\otimes Q_{\G}^{\vee}\otimes Q_{\G})=0$
\item $ \Ho^3(\G, \Sym_3(Q_{\G})\otimes N^{-6}\otimes R^{-6}\otimes Q_{\G}^{\vee})=0$
\item $\Ho^2(\G, \Sym_3(Q_{\G})\otimes N^{-6}\otimes R^{-6}\otimes Q_{\G}^{\vee}\otimes Q_{\G})=0$
\item $\Ho^5(\PP^5, \Sym_2(K)(-1))=0$
\item $ \Ho^4(\G, Q_{\G}\otimes N^{-6}\otimes R^{-6})=0$
\item $ \Ho^3(\G, N^{-6}\otimes R^{-6} \otimes Q_{\G}\otimes Q_{\G}^{\vee})=0$
\end{enumerate}
\end{lemm}

\begin{proof}
For (i), by Proposition \ref{comprop}, we have a short exact sequence as follows
\[0\rightarrow N\otimes Q_{\G}\otimes R^{-5} \rightarrow Q_{\G}\otimes Q_{\G}\otimes R^{-5}\rightarrow R^{-4}\otimes Q_{\G}\rightarrow 0.\]
We claim $\Ho^0(R^{-4}\otimes Q_{\G})=0$. In fact, it follows from
the exact sequence\[(0=)\Ho^0(\PP^5,\Omega_{\PP^5}(-3))=\Ho^0(N\otimes R^{-4}) \rightarrow \Ho^{0}(R^{-4}\otimes Q_{\G})\rightarrow \Ho^0(R^{-3})(=0)\]
where the first term is equal to zero by the Bott vanishing theorem and \cite[5.7.1]{Compo}, the last term is equal to zero by Lemma \ref{vanish} (i). Therefore, we conclude that $\Ho^0(Q_{\G}\otimes Q_{\G}\otimes R^{-5})=0$ since $\Ho^0(R^{-4}\otimes Q_{\G})=0$.\\
It follows that $\Ho^0(Q_{\G}\otimes N\otimes R^{-5})=0$ from the exact sequence
\[\Ho^0(N^2\otimes R^{-5})\rightarrow \Ho^0(Q_{\G}\otimes N\otimes R^{-5}) \rightarrow \Ho^0(N\otimes R^{-4})\]
where the last term is zero by Lemma \ref{vanish} (i) and the first term is zero by
\[\Ho^0(N^2\otimes R^{-5})=\Ho^0(\PP^5,\Sym_2(K)(-5))=0.\]
In the following, we show that $\Ho^0(\PP^5,\Sym_2(K)(-5))=0$.

The fundamental sequence  \cite[Page 36]{Compo} gives rise to an exact sequence as follows \[0\rightarrow \Ho^0(K\otimes K(-5))\rightarrow \Ho^0(K(-5)^{\oplus 6})(=\Ho^0(\Omega^1(-4))^{\oplus 6}=0).\]It follows that $\Ho^0(K\otimes K(-5))=0$. Moreover, we have a short exact sequence
\begin{equation} \label{sym}
0\rightarrow \wedge^2K(-5)\rightarrow K\otimes K(-5) \rightarrow \Sym_2(K)(-5) \rightarrow 0
\end{equation}
which gives rise to an exact sequence\[(0=)\Ho^0(K\otimes K(-5))\rightarrow \Ho^0(\Sym_2(K)(-5)) \rightarrow \Ho^1(\wedge^2K(-5))(=\Ho^1(\Omega^2_{\PP^5}(-3))=0)\] where the last equality $\Ho^1(\Omega^2_{\PP^5}(-3))=0$ follows from the Bott vanishing theorem. Therefore, we show that \[\Ho^0(\PP^5,\Sym_2(K)(-5))=0.\]
\\
For (ii), since we have $\wedge^2Q_{\G}=N\otimes R$ and a short exact sequence\begin{equation} \label{shortstar1}
0\rightarrow R^{-1}\rightarrow Q_{\G}^{\vee}\rightarrow N^{-1}\rightarrow 0,
\end{equation}
it gives rise to an exact sequence
\[(0=)\Ho^2(\wedge^2 \Sym_3\otimes N^{-6}\otimes R^{-7})\rightarrow \Ho^2(\wedge^2\Sym_3(Q_{\G})\otimes (\wedge^2 Q_{\G})^{-6}\otimes Q_{\G}^{\vee})\]\[ \rightarrow \Ho^2(\wedge^2 \Sym_3\otimes N^{-7}\otimes R^{-6})(=0),\] where the last term is equal to zero by Lemma \ref{vanish} (vi). We prove (ii).\\

For (iii), as in the proof of (ii), the short exact sequence (\ref{shortstar1}) gives rise to an exact cohomology sequence as follows:
\[(0=)\Ho^1(\wedge^2 \Sym_3(Q_{\G})\otimes N^{-6}\otimes R^{-7}\otimes Q_{\G})\rightarrow \Ho^1(\wedge^2\Sym_3(Q_{\G})\otimes (\wedge^2 Q_{\G})^{-6}\otimes Q_{\G}^{\vee}\otimes Q_{\G})\]\[
\rightarrow \Ho^1(\wedge^2 \Sym_3(Q_{\G})\otimes N^{-7}\otimes R^{-6}\otimes Q_{\G})(=0)\]
where the first term is zero by the following exact sequence associated to the short exact sequence in Proposition \ref{comprop}, see Lemma \ref{vanish} (vi), \[(0=)\Ho^1(\wedge^2 \Sym_3(Q_{\G})\otimes N^{-5}\otimes R^{-7})\rightarrow \Ho^1(\wedge^2 \Sym_3(Q_{\G})\otimes N^{-6}\otimes R^{-7}\otimes Q_{\G})\] \[ \rightarrow \Ho^1(\wedge^2 \Sym_3(Q_{\G})\otimes N^{-6}\otimes R^{-6})(=0).\]
Similarly, it follows that the last term\[ \Ho^1(\wedge^2 \Sym_3(Q_{\G})\otimes N^{-7}\otimes R^{-6}\otimes Q_{\G})\]is zero from the exact sequence
\[(0=)\Ho^1(\wedge^2 \Sym_3(Q_{\G})\otimes N^{-6}\otimes R^{-6}) \rightarrow \Ho^1(\wedge^2 \Sym_3(Q_{\G})\otimes N^{-7}\otimes R^{-6}\otimes Q_{\G})\]\[ \rightarrow \Ho^1(\wedge^2 \Sym_3(Q_{\G})\otimes N^{-7}\otimes R^{-5})(=0)\]where the first and the last terms are zeros by Lemma \ref{vanish}(vi). We prove (iii).\\
For (iv), the exact sequence in Proposition \ref{comprop} gives rise to an exact sequence
\[ \Ho^3(\G, \Sym_3(Q_{\G})\otimes N^{-6}\otimes R^{-7})\rightarrow  \Ho^3(\G, \Sym_3(Q_{\G})\otimes N^{-6}\otimes R^{-6}\otimes Q_{\G}^{\vee})\]\[ \rightarrow  \Ho^3(\G, \Sym_3(Q_{\G})\otimes N^{-7}\otimes R^{-6})\] where the first and last terms are zeros by Lemma \ref{vanish} (iv). Therefore, we prove (iv).

For (v), Lemma \ref{pullback} and the exact sequence in Propositon \ref{comprop} give rise to an exact cohomology sequence as follows (denote $\Sym_3(Q_{\G})$ by $\Sym_3$):
\begin{equation}\label{five}
\xymatrix{\ldots \Ho^2(\Sym_3\otimes N^{-6}\otimes R^{-7}\otimes Q_{\G})\ar[r]& \Ho^2(\Sym_3\otimes N^{-6}\otimes R^{-6}\otimes Q_{\G}\otimes Q_{\G}^{\vee})\ar[dl] \\
\Ho^2(\Sym_3\otimes N^{-7}\otimes R^{-6}\otimes Q_{\G}) \ldots.}
\end{equation}
The term \[\Ho^2(\Sym_3\otimes N^{-6}\otimes R^{-7}\otimes Q_{\G})\]is zero by the exact sequence \[(0=)\Ho^3(\Sym_3\otimes N^{-5}\otimes R^{-7})\rightarrow \Ho^2(\Sym_3\otimes N^{-6}\otimes R^{-7}\otimes Q_{\G}) \rightarrow \Ho^2(\Sym_3\otimes N^{-6}\otimes R^{-6})(=0)\]which follows from Lemma \ref{vanish} (iv), Lemma \ref{pullback} and the short exact sequence in Proposition \ref{comprop}.

Similarly, the term \[\Ho^2(\Sym_3\otimes N^{-7}\otimes R^{-6}\otimes Q_{\G})\]in the sequence (\ref{five}) is also zero by the exact cohomology sequence
\[\Ho^2(\Sym_3\otimes N^{-6}\otimes R^{-6})\rightarrow \Ho^2(\Sym_3\otimes N^{-7}\otimes R^{-6}\otimes Q_{\G})\rightarrow \Ho^2(\Sym_3\otimes N^{-7}\otimes R^{-5}).\]
It follows from Lemma \ref{vanish} (iv) that the first and the last terms of this cohomology sequence are zeros.\\

For (vi), a similar exact sequence to (\ref{sym}) gives rise to an exact sequence of cohomology groups
\[\Ho^5(\PP^5,(\wedge^2K)(-1))\rightarrow \Ho^5(\PP^5, (K\otimes K)(-1))\rightarrow \Ho^5(\PP^5, \Sym_2(K)(-1))\rightarrow 0.\]The first term \[\Ho^5(\PP^5,(\wedge^2K)(-1))\]is zero by the Bott vanishing theorem and the exact sequence
\[(0=)\Ho^4(K)\rightarrow \Ho^5(K\otimes K(-1)) \rightarrow \Ho^5(K(-1))^{\oplus 6}(=0)\]which is associated to the short exact sequence
\begin{equation}\label{five2}
K\otimes K(-1)\rightarrow (K(-1))^6\rightarrow K\rightarrow 0.
\end{equation} The short exact sequence (\ref{five2}) follows from the Euler sequence
\[0\rightarrow \Omega_{\PP^5}^1\rightarrow \OO_{\PP^5}(-1)^{\oplus 6}\rightarrow \OO_{\PP^5}\rightarrow 0\]and the fact $K=\Omega_{\PP^5}^1(1)$.\\

For (vii), the exact sequence in Proposition \ref{comprop} gives rise to the exact sequence of cohomology groups as follows
\[\Ho^4(\G, N^{-6}\otimes R^{-7}) \rightarrow \Ho^4(\G, Q_{\G}\otimes N^{-6}\otimes R^{-6}) \rightarrow \Ho^4(\G, N^{-7}\otimes R^{-6}) \]
where the first term is equal to $\Ho^5(\G,N)$ by Serre duality. It follows from \cite[5.7.1]{Compo} and the Bott vanishing theorem that \[\Ho^5(\G,N)=\Ho^5(\PP^5, \Omega_{\PP^5}^1(1))=0.\] The last term $\Ho^4(\G, N^{-7}\otimes R^{-6})$ is $\Ho^5(\PP^5, \Sym_2(K)(-1))$ by Serre duality and \cite[5.7.1]{Compo}. By (vi), we know $\Ho^5(\PP^5, \Sym_2(K)(-1))$ is zero, therefore, we prove (vii).\\

For (viii), the exact sequence in Proposition \ref{comprop} gives rise to an exact sequence of cohomology groups as follows
\[\Ho^3(Q_{\G}\otimes N^{-6}\otimes R^{-7})\rightarrow \Ho^3(N^{-6}\otimes R^{-6}\otimes Q_{\G}\otimes Q_{\G}^{\vee}) \rightarrow \Ho^3(Q_{\G}\otimes N^{-7}\otimes R^{-6}).\]We claim that the first and the last terms are zeros. In fact, the exact sequence in Proposition \ref{comprop} gives rise to the following exact sequences of cohomology groups
\[\Ho^3(N^{-5}\otimes R^{-7})\rightarrow \Ho^3(Q_{\G}\otimes N^{-6}\otimes R^{-7})\rightarrow \Ho^3(N^{-6}\otimes R^{-6})\]
\[\Ho^3(N^{-6}\otimes R^{-6})\rightarrow \Ho^3(Q_{\G}\otimes N^{-7}\otimes R^{-6})\rightarrow \Ho^3(N^{-7}\otimes R^{-5}).\]
The first and the last terms of these two sequences are in the form of $\Ho^6(\G, N^{\geq 0}\otimes R^{*})$ which are equal to $\Ho^6(\PP^5, -)$ by \cite[5.7.1]{Compo}. Therefore, they are zeros. We prove the claim and (viii).
\end{proof}

\section{Deformations of Automorphisms}
Suppose that $X$ is a smooth projective scheme over the Witt ring $W(k)=W$ with relative dimension four.
Recall the following assumption
\begin{Assumption}\label{assumption}
We assume that the Hodge-de Rham spectral sequences of $X/W$ degenerates at $E_1$ and the terms are locally free, so that the Hodge and de Rham cohomology sheaves commute with base change.\\
\end{Assumption}
We have that
\begin{theorem} \cite[Theorem 5.5]{PANL}.\label{corolifting}
Suppose that $X$ is a smooth projective scheme over the Witt ring $W(k)=W$ with relative dimension four. Let $X_0$ be the special fiber over $k$ and $f_0$ be an automorphism of $X_0$. Moreover, we assume that the map \[\cris^4(f_0): \cris^4(X_0/W)\rightarrow \cris^4(X_0/W)\] preserves the Hodge filtrations under the natural identification $\cris^4(X_0/W)\cong \Ho^4_{\DR}(X/W)$. 
If the infinitesmal Torelli Theorem holds for $X_0$, i.e., the cup product\[\Ho^1(X_0,T_{X_0})\hookrightarrow
\mathrm{Hom}(\Ho^q(X_0,\Omega^p_{X_0}),\Ho^{q+1}(X_0,\Omega^{p-1}_{X_0}))\]is injective for some $p$ and $q$ with $p+q=4$, then one can lift the automorphism $f_0$ to an automorphism $f:X\rightarrow X$ over $W$. 
\end{theorem}

Let $X_3$ be a smooth cubic fourfold.
\begin{lemm}\label{fre}
Over complex numbers $\C$, the kernel $\Ker (s)$ of the natural map$$s:\A(X_3)\rightarrow \A(\mathrm{H}^4_{sing}(X_3,\mathbb{Q}))$$ is invariant for any smooth deformation of $X_3$.
\end{lemm}
\begin{proof}
Denote by $t$ the map \[\A(\Fa(X_3))\rightarrow \A(\mathrm{H}_{sing}^2(\Fa(X_3),\Q)).\] We claim that $\Ker (s)\simeq \Ker(t)$.

In fact, we have a canonical map \[j:\A(X_3)=\Aut_L(X_3)\rightarrow \A(\Fa(X_3)),\]see Lemma \ref{re}. It is clear that we have a commutative diagram as follows (cf. \ref{abjcoh}): 
\[\xymatrix{ \Fa(X_3)\ar[d]_{j(g)} & \mathrm{L}\ar@{.>}[d]^h \ar[l]_q \ar[r]^p & X_3\ar[d]^g\\
\Fa(X_3) & \mathrm{L}\ar[l]_q \ar[r]^p & X_3}\]
where $g\in \A(X_3)=\A_L(X_3)$ and $q:\mathrm{L}\rightarrow \Fa(X_3)$ is the universal bundle of lines. We have a morphism \[h:\mathrm{L}\rightarrow \mathrm{L}\]associating to a line $l$ in $X_3$ to the line $g(l)$ in $X_3$. 
Note that the inverse map of $j(g)^*$ is $j(g)_*$. Therefore, we have \[j(g)^* \circ (q_*p^*)=q_*\circ h^*\circ p^*=(q_*p^*)\circ g^*\]on $\Ho^4_{sing}(X_3,\Q)$. Since the Abel-Jacobi map \[q_*p^*:\Ho^4_{sing}(X_3,\Q)\rightarrow \Ho^2_{sing}(\Fa(X_3),\Q)\] is an isomorphism (\cite{Bea2}), the following are equivalent:
\begin{itemize}
\item $j(g)^*=\Id$ on $\Ho^2_{sing}(\Fa(X_3),\Q)$.
\item $g^*=\Id$ on $\Ho^4_{sing}(X_3,\Q)$.
\end{itemize}
In particular, the map $j$ sends $\Ker(s)$ into $\Ker(t)$. We claim this map is an isomorphism. In fact, the injectivity is obvious by the fact that, for each point $x\in X_3$, there are at least two lines in $X_3$ passing through $x$.

We show that $j$ is surjective as follows. Since $\Fa(X_3)$ is simply connected, an automorphism $j(g)$ such that $j(g)^*=\Id$ on $\Ho^2_{sing}(\Fa(X_3))$ must preserve the natural ample bundle $\OO_{\Fa(X_3)}(1)$ of $\Fa(X_3)$ which is induced by the Pl\"ucker embedding.

We claim that an autorphism $f$ of $\Fa(X_3)$ preserving the natural polarization $\OO_{\Fa(X_3)}(1)$ is coming from a morphism \[g:X_3\rightarrow X_3,\] in other words, we have $j(g)=f$. In fact, we consider the natural inclusions \[\Fa(X_3)\subseteq \mathbb{G}(1,5)\subseteq \PP^N\]where the last inclusion is the Pl\"ucker embedding. By \cite[1.16(iii)]{Compo}, the Grassimanian $\mathbb{G}(1,5)$ is the intersection of the quadric
hypersurfaces in $\PP^N$ containing $\Fa(X_3)$ . So there is an automorphism $g': \mathbb{G}(1,5)\rightarrow \mathbb{G}(1,5)$ sending $\Fa(X_3)$ to $\Fa(X_3)$ and such that $g'|_{\Fa(X_3)}=f$. By a theorem of Chow \cite{Chow}, we have \[\Aut(\mathbb{G}(1,5))=\Aut(\PP^5).\]So there is a linear automorphism $g''$ of $\PP^5$ sending $X_3$ to $X_3$. The automorphism $g''|_{X_3}$ induces the automorphism \[f:\Fa(X_3)\rightarrow \Fa(X_3).\] It follows that $j(g''|_{X_3})=f.$ We show that $\Ker (s)\simeq \Ker(t)$.

Since $\Fa(X_3)$ is a holomorphic symplectic manifold (\cite{Bea2}), the lemma follows from \cite[Theorem 2.1]{BT}.

\end{proof}
\begin{prop} \label{injective}
Over complex numbers $\mathbb{C}$, the map
\[\A(X_3)\rightarrow \A(\Ho^4_{sing}(X_3,\Q))\] is injective.
\end{prop}
\begin{proof}
The proposition follows from the previous lemma and the fact that the automorphism group of a general cubic fourfold is trivial, see \cite{MM}. For a different proof of the proposition, we refer to \cite[Proposition 6.1]{PAN}.
\end{proof}

\begin{remark}
One can also use \cite[Theorem 3.2 and Corollary 3.3]{PANL} to give a simple proof.
\end{remark}

\begin{theorem} \label{mainthmmap}
Suppose that $X_3$ is a smooth cubic fourfold over an algebraically closed field $k$. Let $l$ be a prime different from the characteristic of $k$. We have the following injections:
\begin{itemize}
\item $\Aut(X_3)\hookrightarrow \Aut(\et^4(X_3,\Ql))$
\item $\Aut(\Fa(X_3))\hookrightarrow \Aut(\et^2(\Fa(X_3),\Ql)).$
\end{itemize}
\end{theorem}
\begin{proof}
By Proposition \ref{injective} and the proof of Lemma \ref{fre}, the theorem holds in characteristic zero. Assume the characteristic of $k$ is positive. Let $f_0\in \A(X_3)$ be an automorphism of $X_3$ such that $f_0^*=\Id$ on $\et^4(X_3,\Ql)$. It follows from Lemma \ref{idladic} that $f_0^*=\Id$ on $\cris^4(X_3/W)_K$ where $K$ is the fraction field of the Witt ring $W=W(k)$. Therefore, by Lemma \ref{cristorsionfree} , we have $f_0^*=\Id$ on $\cris^4(X_3/W)$. Let $\widehat{X_3}$ be a lifting of $X_3$ over $W$. By the comparision theorem of crystalline cohomology and de Rham cohomology, we have
\[f_0^*=\Id:\Ho^4_{\DR}(\widehat{X_3}/W)\rightarrow \Ho^4_{\DR}(\widehat{X_3}/W).\]
In particular, the action $f_0^*$ preserves the Hodge filtractions of $$\Ho^4_{\DR}(\widehat{X_3}/W)=\cris^4(X_3/W).$$ By Theorem \ref{corolifting} and the infinitesimal Torelli Theorem \cite[Theorem 3.4]{Tate} for cubic fourfolds, one can lift the map $f_0$ to a map $\hat{f}$ on $\widehat{X_3}/W$. By Proposition \ref{injective}, we conclude that \[(\hat{f})_K:(\widehat{X_3})_K\rightarrow (\widehat{X_3})_K\]is the identity map. Therefore, its specialization $f_0$ is the identity map where $K$ is the fractiion field of $W(k)$. We prove the first statement. 

For the second statement, by Proposition \ref{nsgroup}, we know the Picard group \[\Pic(\Fa(X_3))=\NS(\Fa(X_3))\] of $\Fa(X_3)$ is torsion-free, hence, it is a subgroup of $\et^2(\Fa(X_3),\Ql)$ where we omit the tate twist. Hence, for an automorphism $f$ such that $f^*=\Id$ on $\et^2(\Fa(X_3),\Ql)$, we have $f^*=\Id$ on the Picard group $\Pic(\Fa(X_3))$. Let $L$ be the very ample line bundle of $\Fa(X_3)$ induced by Pl\"ucker embedding. We have $f^*(L)=L$. By the proof of Lemma \ref{fre}, we know $f$ is induced by an automorphism $g$ of the cubic fourfold $X_3$. To prove $f$ is the identity map, it suffices to show that $g$ is the identity map. 

In fact, it follows from $\et^2(\Fa(X_3),\Ql)\simeq \et^4(X_3,\Ql)$ \cite{Bea2} that $g^*=\Id$ on  $\et^4(X_3,\Ql)$. Hence, we have $g=\Id$ by the first statement. We prove the theorem.
\end{proof}

\section{Lefschetz Cubics}
Let $k$ be an algebraically closed field. We assume that $\chari(k)\neq 2,3$. We call a cubic hypersurface is a Lefschetz cubic if it has exactly one ordinary double point.
\begin{lemm}\label{picone}
Let $X_d$ be a hypersurface in $\PP^{n}$ over $k$. If the dimension of $X_d$ is at least $3$, then we have\[\Aut_L(X_d)=\A(X_d).\]
\end{lemm}
\begin{proof}
By \cite[Chapter IV.3, Corollary 3.2]{Hart}, we know \[\Pic(X_d)=\mathbb{Z}[\OO_{X_d}(1)].\] It is clear that $\Ho^0(X_d,\OO_{X_d}(1))=\Ho^0(\PP^n,\OO_{\PP^n}(1))$. Therefore, we conclude \[\Aut_L(X_d)=\A(X_d).\]
\end{proof}

\begin{thm} \label{node}
Let $X_3$ be a Lefschetz cubic threefold (resp. fourfold) over $k$. Then, the natural map
\[\A (X_3)\rightarrow \A(\et^3(X_3,\Ql))\](resp. $\A (X_3)\rightarrow \A(\et^4(X_3,\Ql))$ ) is injective.
\end{thm}
\begin{proof}
We give a proof for Lefschetz cubic fourfolds and sketch a proof for Lefschetz cubic threefolds. First of all, we blow up the node $p$ of the cubic fourfold $X_3$. Denote by \[\Bl_p:\X \rightarrow X_3\] the blow-up. Suppose $g\in \A(X_3)$. It is clear that $g(p)=p$. Therefore, it induces a map $g'$ from $\X$ to $\X$. We have the following diagram
\begin{equation}\label{com}
\xymatrix{Q \ar[d]^{g'|_Q} \ar@{^{(}->}[r]^i& \X \ar@{.>}[d]^{g'} \ar[r]^{\Bl_p} & X_3\ar[d]^g\\
Q \ar@{^{(}->}[r]^i& \X \ar[r]^{\Bl_p} & X_3}
\end{equation}
where $Q$ is the exceptional divisor of the blow-up, it is a smooth quadric threefold \cite[Expos$\acute{e}$ XII]{SGA7}.
We apply the long exact sequence of l-adic cohomology with compact support to the pair $(X_3,X_3-p)$. Denote $X_3-p$ by $U$. We have an exact sequence as follows
\[\xymatrix{ (0=)\et^3(p) \ar[r] & \etc^4(U)\ar[r] & \et^4(X_3)\ar[r]& \etc^4(p)(=0) }.\]
We conclude that
\[\etc^4 (U) \simeq  \et^4(X_3).\]
On the other hand, we have an exact sequence as follows
\[(0=)\et^3(Q)\rightarrow \et^4(X_3) \rightarrow \et^4(\X)\oplus \et^4(p) \rightarrow \et^4(Q)\rightarrow \et^5(X_3),\]see~\cite[Page 230]{Milne}.
The exact sequence gives rise to an exact sequence
\begin{equation}\label{etstar}
0\rightarrow \et^4(X_3) \rightarrow \et^4(\X)\rightarrow \et^4(Q)(=\Ql),
\end{equation}
see~\cite[Expos$\acute{e}$ XII.3]{SGA7}.
Let $\N_{Q/\X }$ be the normal bundle of $Q$ in $\X$. By \cite[VII,Theorem 4.1]{SGA5}, we have maps $i_*$ and $i^*$ such that \[i^*i_*(-)=(-)\cup c_1(\N_{Q/\X})\]and
\begin{equation}\label{spl}
\xymatrix{\Ql=\et^2(Q) \ar@/^1pc/[rr]^{-\cup c_1(\N_{Q/\X})} \ar[r]_{i_*} &\et^4(\X)\ar[r]_{i^*} &\et^4(Q)=\Ql}
\end{equation}
where $\Ql=\et^2(Q)$ and $\et^4(Q)=\Ql$ follow from \cite[XII]{SGA7}. Let $\widetilde{\PP^5}$ be the blow-up of $\PP^5$ at $p$. Since $\X$ is the strictly transformation of $X_3$ in the blow-up $\widetilde{\PP^5}$, it is easy to see $\N_{Q/\X}=\OO_{Q}(-1)$. Therefore, the map $i^*i_*$ is an isomorphism which gives a spliting by (\ref{etstar}). In particular, we have
\[\xymatrix{\et^2(Q)\oplus \et^4(X_3) \ar[r]^<<<<<{(i_*,\Bl_p^*)}_<<<<{\simeq} &\et^4(\X)}.\]
If $g^*=\Id$ on $\et^4(X_3)$, then, by Lemma \ref{node1} below and the fact $\et^2(Q)=\Ql<c_1(\OO_Q(1))>$ (see \cite[XII]{SGA7}), we conclude that the map $g'^*=\Id$ on $\et^4(\X)$.

Denote by $S$ the space of lines in $X_3$ passing through $p$. Let \[L\rightarrow S\] be the tautological bundle of lines. We hope there is no confusion if we also denote by $L$ the union of lines in $X_3$ passing through $p$. In an affine coordinate system with the origin $p$, we can write down the equation for $X_3$ as follows:
\[0=f_2(x_1,\ldots,x_5)+f_3(x_1,\ldots,x_5)\]
where $f_2$ (resp. $f_3$) is a polynomial of degree $2$ (resp. $3$). Therefore, the surface $S$ is a complete intersection in $\PP^4$ of type $(2,3)$ defined by the equations $f_2$ and $f_3$. Since $X_3$ has exactly one node, the surface $S$ is a smooth $K3$ surface by local calculations. We project $X_3$ from $p$ to $\PP^4$. It gives rise to the following commutative diagram (see Lemma \ref{picone}):
\[\xymatrix{\X\ar[d]^{\simeq} \ar[r]^{\Bl_p}& X_3 \ar@{.>}[d]^{\proj} &X_3-L \ar[d]^{\simeq} \ar@{_{(}->}[l]\ar[r]^{g|_{X_3-L}} &X_3-L\ar[d]^{\simeq}\\
\Bl_S(\PP^4)\ar[r]^{\Bl_S} &\PP^4& \PP^4-S\ar@{_{(}->}[l] \ar[r]^{g_0} &\PP^4-S}\]
where $\proj$ is the projection from $p$ to a projective subspace $\PP^4$ and $g_0$ is induced by the vertical identification between $X_3-L$ and $\PP^4-S$. Since $S$ is of codimension $2$, the map $g_0$ extends to an automorphism $g_1$ of $\PP^4$ sending $S$ to $S$. It induces a map \[g_2:\Bl_S(\PP^4)\rightarrow \Bl_S(\PP^4).\] Under the identification $\Bl_S(\PP^4)=\X$, the map $g':\X \rightarrow \X$ coincides with the map $g_2$ over the open subset $\Bl_S(\PP^4)-E$ where $E$ is the exceptional divisor of the blow-up $\Bl_S(\PP^4)$ along $S$. It follows that $g_2=g'$. So we have the following diagram:
\[\xymatrix{E \ar[d]^{\pi} \ar@{^{(}->}[r]^{j} & \Bl_S(\PP^4)\ar[d]^{\Bl_S} \ar[r]^{g_2} &\Bl_S(\PP^4)\ar[d]^{\Bl_S} &E\ar@{_{(}->}[l]_<<<<j \ar[d]^{\pi}\\
S \ar@{^{(}->}[r]^{inc}& \PP^4 \ar[r]^{g_1} &\PP^4& S\ar@{_{(}->}[l]}\]
Let $\varphi$ be a map as follows:
\begin{equation}\label{star2}
\et^4(\PP^4)\oplus \et^2(S) \oplus \et^0(S) \rightarrow \et^4(\Bl_S(\PP^4))
\end{equation}
maps $(x,y,z)$ to $\Bl_S^*(x)+j_*\pi^*(y)+j_*((\pi^*z)\cup c_1(\OO_E(1)))$. It follows from \cite[Theorem 7.3.1]{Voisinhodge} and \cite[VII]{SGA5} that the map $\varphi$ is an isomorphism.

By Lemma \ref{node2} below, we have that
\begin{equation}\label{star1}
g_2^*\circ \varphi=\varphi \circ (g_1^*,g_1|_S^*, g_1|_S^*).
\end{equation}
From the discussion above, we know $g_2^*=\Id$ on $\et^4(\Bl_S(\PP^4))$ if the map $g^*=\Id$ on $\et^4(X_3)$. Therefore, we have $(g_1|_S)^*=\Id$ on $\et^2(S,\Ql)$. By \cite[Proposition 3.4.2]{Riz}, we conclude that $g_1|S$ is the identity since $\et^2(S,\mathbb{Z}_l)$ is torsion-free. Since $S$ is a smooth complete intersection surface of type $(2,3)$ in $\PP^4$, we have \[\Ho^0(S,\OO_S(1))=\frac{c_1(\OO_S(1))\cdot c_1(\OO_S(1))}{2}+\chi(\OO_S)=3+2=5\]by the Riemann-Roch theorem. Therefore, we know the inclusion $inc$ is given by the complete linear system $|\OO_S(1)|$
\[\xymatrix{S \ar@/^1pc/[rrr]^{g_1|_S} \ar[r]_{inc}& \PP^4\ar[r]^{g_1} &\PP^4& S\ar[l]^{inc}}.\] Therefore, we conclude $g_1=\Id$ and $g_2=\Id$. In summary, we prove the theorem holds for Lefschetz cubic fourfolds. For the cubic threefolds, the proof is similar. We sketch a proof here.
\begin{enumerate}
  \item The exact sequence (\ref{etstar}) is replaced by \[\et^3(X_3)\rightarrow \et^3(\X)\rightarrow 0.\]Then $g_2^*=\Id$ on $\et^3(\X)\cong \et^3(\Bl_S\PP^3)$.
  \item The space $S$ of lines is a smooth complete intersection curve of type $(2,3)$ in $\PP^3$. Instead of using \cite[Proposition 3.4.2]{Riz}, we use \cite[Theorem 1.13]{DM} to show that $g_1=\Id$ and $g_2=\Id$.
\end{enumerate}
(i) and (ii) imply the theorem holds for Lefschetz cubic threefolds.
\end{proof}

\begin{lemm}\label{node1}
Under the hypothesis as above, we have the following commutative diagram:
\[\xymatrix{\et^4(X_3)\oplus\et^2 (Q) \ar[d]_{(g^*,g'^*|_{Q})} \ar[r]^<<<<{(\Bl_p^*,i_*)}_<<<{\simeq} & \et^4(\X)\ar[d]^{g^*}\\
\et^4(X_3)\oplus \et^2(Q) \ar[r]^<<<<{(\Bl_p^*,i_*)}_<<<{\simeq} &\et^4(\X).}\]
\end{lemm}
\begin{proof}
By the diagram (\ref{com}), we have
\begin{itemize}
  \item $g'^*\circ \Bl_p^*=\Bl_p^*\circ g^*$ and
  \item $g'^*\circ i_*=i_*\circ g'|_Q^*$ since $g'^*$ is the inverse of $g'_*$. We prove the lemma.
\end{itemize}
\end{proof}
\begin{lemm}\label{node2}
With the same notations and hypotheses as above, we have \[g_2^*\circ \varphi=\varphi \circ (g_1^*,g_1|_S^*, g_1|_S^*),\]see (\ref{star1}) and (\ref{star2}).
\end{lemm}
\begin{proof}
To prove the lemma, it suffices to show that \[g_2^*(\Bl_S^*(x)+j_*\pi^*(y)+j_*((\pi^*z)\cup c_1(\OO_E(1))))\]is equal to \[\Bl_S^*g_1^*(x)+j_*(\pi^*(g_1|_S)^*(y))+j_*[(\pi^*(g_1|_S)^*(z))\cup c_1(\OO_E(1))].\]
We claim that
\begin{itemize}
  \item $g_2^*\circ \Bl_S^*(x)=\Bl_S^*\circ g_1^*(x)$,
  \item $g_2^*\circ j_*\circ \pi^*(y)=j_*\circ \pi^*\circ (g_1|_S)^*(y)$,
  \item $(g_2^*\circ j_*)(\pi^*z\cup c_1(\OO_E(1)))=j_*((\pi^*\circ (g_1|_S)^*(z))\cup c_1(\OO_E(1)))$.
\end{itemize}
The first equality is obvious. The second equality follows from the equality\[g_2^*\circ j_*(\pi^*y)=\left(j_*\circ (g_2|_E)^*\right)(\pi^*y)\] and $g_{2*} =(g^*_2)^{-1}$. To show the third equality, it suffices to show that \[(g_2^*\circ j_*)(\pi^*z\cup c_1(\OO_E(1)))=j_*([(g_2|_E)^*(\pi^*(z))]\cup c_1(\OO_E(1))).\]In fact, it follows that\[g_{2*}j_{*}((g_2|_E)^*(\pi^*(z))\cup c_1(\OO_E(1)))\]
is equal to
\[j_*((g_2|_E)_*([(g_2|_E)^*(\pi^*z)]\cup c_1(\OO_E(1))))=j_*(\pi^*(z)\cup [(g_2|_E)_*(c_1(\OO_E(1)))])\]
from the projection formula and the fact that $g_{2*} =(g^*_2)^{-1}$. Therefore, it suffices to show that
\[j_*(\pi^*z\cup c_1(\OO_E(1)))=j_*(\pi^*(z)\cup [(g_2|_E)_*(c_1(\OO_E(1)))]).\]
We claim that $(g_2|_E)_*(c_1(\OO_E(1)))=c_1(\OO_E(1))$. The lemma follows from the claim. To prove the claim, we just need to show
\[(g_2|_E)^*(\OO_E(1))=\OO_E(1).\]
In fact, we have $\OO_E(-1)=N_{E/\Bl_S(\PP^4)}$ where $N_{E/\Bl_S(\PP^4)}$ is the normal bundle of $E$ in $\Bl_S(\PP^4)$. So it gives \[(g_2|_E)^*(N_{E/\Bl_S(\PP^4)})=N_{E/\Bl_S(\PP^4)}.\]We prove the claim.

\end{proof}

\section{Vanishing Cycles and the Middle Picard Number}
In this section, we apply the theory of Milnor fibers and vanishing cycles to show that there is a smooth cubic fourfold of middle Picard number one with an automorphism of order $3$, see Proposition \ref{pic}. All varieties in this section are over complex numbers $\C$. We fix some notations as follows.

Suppose that $F_1(X_1,X_2,\ldots,X_5)$ and $F_2(X_1,X_2,\ldots,X_5)$ are cubic forms defining smooth cubic threefolds. Let $V_1$ (resp. $V_2$) be the smooth cubic threefold in $\PP^4$ defined by $F_1=0$ (resp. $F_2=0$). By the theory of Lefschetz pencils, we can assume that $F_1$ and $F_2$ are two forms such that the intersection $V_1\cap V_2$ is smooth and every fiber of the family $X$ of cubic threefolds defined by \[F_1(X_1,\ldots, X_5)+tF_2(X_1,\ldots,X_5)=0, t\in \C\]has at most one node.

 Using this family $X$ of cubic threefolds, we can construct a family $X_{\ao}$ of cubic fourfolds \[f:X_{\ao}\rightarrow \ao\] as follows: 
 
 $X_{\ao}$ is given by
 \begin{equation}\label{cubeq}
\{(t,[X_0:,\ldots,:X_5])\in \ao\times \PP^5| X_0^3+F_1(X_1,\ldots,X_5)+tF_2(X_1,\ldots,X_5)=0\}.
\end{equation}
 In other words, the fiber of $X_{\ao}\rightarrow \ao$ over $b\in \C$ is a cubic fourfold which is a 3-cyclic branched covering over $\PP^4$ along the cubic threefold $(X)_b$ which is the cubic threefold in $X$ over $b$.
Hence, every fiber of $X_{\ao}\rightarrow \ao$has at most one singularity and the singularity must be of $A_2$-type. Such an $A_2$-singularity is locally defined by the equation
\[X_0^3+X_1^2+X_2^2+\ldots+X_5^2=0.\]
Denote by $x_1,\ldots,x_m \in X_{\ao}$ the $A_2$-singularities in $X_{\ao}$. The Milnor fiber $\Mil_{f,x_i}$ corresponding to $x_i$ is homotopic to $ S^4\vee S^4$ since its Milnor number is two.

The family $X_{\ao}\rightarrow \ao$ can be completed to a pencil of cubic fourfolds over $\PP^1$ naturally. Denote the pencil by \[f:X_{\PP^1}\rightarrow \PP^1.\] In particular, it is the pencil of cubic fourfolds defined by cubic fourfolds $X_0^3+F_1=0$ and $F_2=0$. Over the point $\infty\in \PP^1$, the fiber $X_{\infty}$ is the cubic fourfold defined by $F_2(X_1,\ldots, X_5)$. It is clear that the cubic fourfold $X_{\infty}$ is a cone. 
An easy calculation shows that the base locus $B$ of this pencil is smooth since $V_1\cap V_2$ is smooth. Moreover, we have $X_{\PP^1}=\Bl_B(\PP^5)$. 

Let $U=\ao-\{f(x_1),\ldots,f(x_m)\}$ and $X_b=f^{-1}(b)$ ($b\in U$). It is clear that $X_{\PP^1}-X_{\infty}=X_{\ao}$. We have the following lemma to describe the vanishing cycles.

\begin{lemm}\label{vc}\label{generate}
Let $Y$ be a smooth fiber of $X_{\ao}\rightarrow \ao$ over some point $b\in U$. Denote by $j'$ (resp. $j$) the inclusion $Y \subseteq \PP^5$ (resp. $Y=f^{-1}(b) \subseteq X_{\PP^1}=\Bl_B(X)$). Then, 
\begin{enumerate}
\item $\Ho^4_{prim}(Y,\C)=\Ho^4_{van}(Y,\C):=\Ker(j'_*:\Ho^4(Y,\C)\rightarrow \Ho^6(\PP^5,\C))$.
\item By $\Poincar$\text{\'e} duality, the space $\Ho^4 _{van}(Y,\C)$ of vanishing cycles is generated by the image of $\Ho_4(\Mil_{f,x_i},\mathbb{C})$ in \[\Ho_4(X_{t_i},\C)\cong \Ho_4(Y,\C),\] where $t_i$ is a regular value of $f$ nearby $f(x_i)$ and we identify the homology and cohomology groups of $X_{t_i}$ with $Y's$.
\end{enumerate}
\end{lemm}

\begin{proof}
The first statement follows from \cite[Prop 2.27]{Voisinhodge}.  For the second statement, we claim that
\[\Ker(j'_*:\Ho_4(Y,\C)\rightarrow \Ho_4(\PP^5,\C))=\Ker(j_*:\Ho_4(Y,\C)\rightarrow \Ho_4(X_{\PP^1}-X_{\infty},\C))\]where the left hand side is $\Ho_{van}^4(Y,\C)$ by $\Poincar$\'e duality. In fact, since $X_{\PP^1}=\Bl_B(\PP^5)$ and $j'=\Bl|_{X_{\ao}}\circ j$ where $\Bl$ is the blow-up morphism \[\Bl:X_{\PP^1}=\Bl_B(\PP^5)\rightarrow \PP^5.\] So we have a diagram as follows
\[\xymatrix{0 \ar[r] & \Ker (j_*)\ar[r] \ar[d]^{h} &\Ho_4(Y)\ar@{=}[d] \ar[r]^{j_*} &\Ho_4(X_{\ao})\ar[d]_{(Bl|_{X_{\ao}})_*}\\
0\ar[r]& \Ker(j'_*)\ar[r] &\Ho_4(Y) \ar[r]^{j'_*} &\Ho_4(\PP^5)}\]
where $X_{\ao}=X_{\PP^1}-X_{\infty}$. If $j_*$ is surjective, then $h$ is an isomorphism by the snake lemma. To show the surjectivity of $j_*$, we consider the long exact sequence of homology groups associated to $(X_{\ao},Y)$
\[\xymatrix{\Ho_5(X_{\ao},Y) \ar[r] &\Ho_{4}(Y) \ar[r]^{j_*} & \Ho_4(X_{\ao})\ar[r] &\Ho_4(X_{\ao},Y)}.\] By Lemma \ref{lemm5} below, we have $\Ho_{4}(X_{\ao}, Y)=0.$ Therefore, we show that $j_*$ is surjective. The lemma follows.
\end{proof}

\begin{lemm}\label{lemm5}
With the notations as above, we have
\[\Ho_j(X_{\ao},Y)\simeq \oplus_{i=1}^m \Ho_j(X_{\Delta_i},Y)\simeq \oplus_{i=1}^m \widetilde{H}_{j-1}(\Mil_{f,x_i})\]
where $\Delta_i$ is a small disk containing $f(x_i)$ and we identify $Y$ with any smooth fiber of $f$ so that the relative homology groups make sense. Moreover, we have $\Ho_{4}(X_{\ao}, Y)=0$ and $\Ho_{5}(X_{\ao}, Y)=\oplus_{i=1}^m \Ho_{4}(\Mil_{f,x_i})$.
\end{lemm}
\begin{proof}
In fact, we can choose a good cover $\{U_1,\ldots,U_{s+1}\}$ of $\PP^1$ such that
\begin{itemize}
\item the finite intersections of $\{U_i\}$ are contractible,
\item the intersection $\cap_{i\in I} U_i$ does not contain any point of $\{x_1,\ldots,x_m,\infty\}$ if $\#(I)$ ($I\subseteq \{1,2,\ldots,s+1\}$) is greater than one,
\item $\PP^1-(U_1\cup\ldots\cup U_s)$ contains $\infty$ and it is contractible to $\infty$,
\item $U_i$ contains $x_i$ and $U_j$ does not contain any $x_i$ if $j\geq m+1$.
\end{itemize}
It follows from a deformation retraction that the inclusion $(X_{U_1\cup\ldots\cup U_s},Y)\subseteq (X_{\ao},Y)$ is a homotopy equivalence. 
It follows that \[\Ho_j(X_{\ao},Y)\simeq \Ho_j(X_{U_1\cup\ldots\cup U_s},Y).\] By the Mayer Vietoris sequence and the induction on the number of the coverings $\{U_i\}$, we have $ \Ho_j(X_{U_1\cup\ldots\cup U_s},Y)= \oplus_{i=1}^m \Ho_j(X_{U_i},Y)$. Since we can shrink $U_i$ to a small disk $\Delta_i$, we show that\[\Ho_j(X_{\ao},Y)\simeq \oplus_{i=1}^m \Ho_j(X_{\Delta_i},Y).\]
By the excision theorem (see \cite[Proof of Lemma C.13]{MHS}), we have \[\Ho_*(X_{\Delta_i},Y)\simeq \Ho_*(B,\Mil_{f,x_i})\simeq \widetilde{H}_{*-1}(\Mil_{f,x_i})\]where $B$ is a small ball that is centered at $x_i$. We prove the lemma.
\end{proof}

In the following, we denote $\Ho^0(\PP^4,\OO_{\PP^4}(3))$ by $V$. Suppose that $V$ has a basis \[\{F_1,F_2,\ldots,F_{N+1}\}.\] Using this basis, we identify $\PP(V)$ with $\PP^N.$ Let $[t_1:\ldots :t_{N+1}]$ be the coordinates of $\PP^N$. So we have a natural inclusion
\begin{equation} \label{nz}
\ao\hookrightarrow \mathbb{A}^{N-1}=\{t_1\neq 0\}
\end{equation}
which maps $t\in \ao$ to $[F_1+tF_2]$.

By the theory of Lefschetz pencils, we know there is an irreducible discriminant hypersurface $H'$ in $\PP(V)$ parametrizing singular cubic threefolds, see \cite[Chapter 2(2.2), Chapter 3 3.2.2]{Voisinhodge}. Denote the universal smooth hypersurface by\[\varphi: \mathcal{X}\rightarrow \PP(V)-H'.\] Let $H'_o$ be the open subset of $H'$ parametrizing cubic threefolds with at most one ordinary double point. The subvariety $H'_o$ is smooth, see \cite[Chapter 2, Lemma 2.7 and Proposition 2.9]{Voisinhodge}. Let $H$ be $H'\cap \mathbb{A}^{N-1}$ and $\Ho_o$ be $H'_o\cap \mathbb{A}^{N-1}$ where $\mathbb{A}^{N-1}=\{t_1\ne 0\}$, see (\ref{nz}). In particular, we have $\Ho_o\cap \ao=\{x_1,\ldots,x_m\}$. Assume $w=\mathrm{exp}(\frac{2\pi i}{3})$.

\begin{lemm}\label{conjugate}
Let $r_i$ be the local monodromy operator of $\pi_1(\Delta_i^*)$ on the homology $\Ho_4(\Mil_{f,x_i})$. We have the following decomposition
\[\Ho_4(\Mil_{f,x_i},\C)=\Ho_{w}^i\oplus \Ho_{w^2}^i\]
where $\Ho_{w}^i$ (resp. $\Ho_{w^2}^i$) is the $w$(resp. $w^2$)-eigenspace of $r_i$ and of dimension one. Denote by $\delta_w^i$ (resp. $\delta_{w^2}^i$) the corresponding eigenvector. Under the monodromy action of $\pi_1(U)$, all the image of $\delta_w^i$ (resp. $\delta_{w^2}^i$) in $\Ho^4_{prim}(X_b,\C)=\Ho^4_{van}(X_b,\C)$ are conjugate up to a nonzero scalar where $b$ is a point in $U$.
\end{lemm}
\begin{proof}
Since the singularity $x_i$ is of $A_2$-type, i.e., it is locally as
\[X_0^3+X_1^2+\ldots+X_5^2=0,\]the first statement follows from the Brieskorn-Pham theorem and the fact that $\Mil_{f,x_i}\simeq S^4\vee S^4$, see \cite[Theorem 9.1]{milnor}.
For the second statement, we notice that $H=H'\cap \mathbb{A}^{N-1}$ is irreducible, so its smooth open subset $\Ho_o$ is connected. For points $x_i$ and $x_j$ in $\Ho_o$, we can choose a path $l$ from $x_i$ to $x_j$ in $\Ho_o$. We lift $l$ to a path $l'$ from $x_i'$ to $x_j'$ in the boundary of a tubular neighbourhood of $\Ho_o$ in $\mathbb{A}^{N-1}$, where $x_i'$ (resp. $x_j'$) is a point very closed to $x_i$ (resp. $x_j$). Obviously, a trivialisation of $\varphi$ over $l'$ transports the vanishing cycle classes generated by the image of $\Ho_4(\Mil_{f,x_i})$ to the vanishing cycle classes generated by the image of $\Ho_4(\Mil_{f,x_j})$, see Lemma \ref{generate}. Under this transportation, the monodromy operators $r_i$ and $r_j$ are conjugate by the path $l'$. So we identify $\Ho_{w}^i$  (resp. $\Ho_{w^2}^i$) with $\Ho_{w}^j$ (resp. $\Ho_{w^2}^j$) by the monodromy action. We prove the lemma.
\end{proof}

Consider $f|_{X_U}:X_U\rightarrow U$, see (\ref{cubeq}). There is a natural $\Z/3\Z$ action on $X_U$ fiberwisely. More precisely, the action of the generator of $\Z/3\Z$ is given by
\[(t,[X_0:X_1:\ldots:X_5])\rightarrow (t,[w X_0:X_1:\ldots:X_5])\]
where $w$ is $\mathrm{exp}(\frac{2\pi i}{3})$. This action induces an action on the local system $(R^4 f_*(\C))_0$ (over $U$) whose stalk over $b\in U$ is $\Ho^4_{prim}(X_b,\C)$. So we have
\[\Ho^4_{prim}(X_b,\C)=\Ho^4_w\oplus \Ho^4_{w^2}\]
where  $\Ho^4_w$ (resp. $\Ho^4_{w^2}$) is the $w$ (resp. $w^2$)-eigenspace of the $\Z/3\Z$-action. The eigenspace $\Ho^4_w$ (resp. $\Ho^4_{w^2}$) contains $\Ho^{3,1}$ (.resp $\Ho^{1,3}$) by the Griffiths residues, see (\ref{residue}). Since the $\Z/3\Z$-action commutes with the monodromy action of $\pi_1(U)$, the summands $\Ho_w^4$ and $\Ho_{w^2}^4$ are $\pi_1(U)$-modules. In particular, the summand $\Ho_w^4$ (resp. $\Ho_{w^2}^4$) is generated by the vanishing cycles $\{\delta_w^i\}$ (resp. $\{\delta_{w^2}^i\}$) for $i=1,2,\ldots,m$, see Lemma \ref{generate} and \cite[Theorem 9.1]{milnor}.

We use the Picard-Lefschetz transformation formula to show the following proposition.
\begin{prop}\label{irr}
With the notations as above, the monodromy representation $\pi_1(U)$ of $\Ho^4_{prim}(X_b,\C)$ is complete reducible. The irreducible decomposition is \[\Ho^4_{prim}(X_b,\C)=\Ho^4_w\oplus \Ho^4_{w^2}.\] 
\end{prop}
\begin{proof}
Let $W$ be a non-zero subrepresentation of $\Ho^4_{prim}(X_b)$. We claim that $W$ contains $\Ho^4_w$ or $\Ho^4_{w^2}$. In fact, it follows from Lemma \ref{vc} that  $\Ho^4_{prim}(X_b)$ is generated by $\delta_w^i$ and $\delta_{w^2}^i$. For a nonzero element $y\in W$, it follows that $y=\sum\limits_{i=1}^m y_i$ where \[y_i\in \mathrm{Im}(\Ho^4(\Mil_{f,x_i}))=\mathrm{Span}(\delta_w^i,\delta_{w^2}^i).\]It is clear that the local monodromy operator $r_j$ fixes $\mathrm{Im}(\Ho^4(\Mil_{f,x_i},\C))$ if $i\neq j$.

We use the notations following \cite[Chapter 3]{singhyper}. On the space $\Ho_4(\Mil_{f,x_i})$, there is an intersection form $<,>_i$. It is nondegenerate on $\Ho_4(\Mil_{f,x_i})$ since the singularity is of $A_2$-type, see \cite[Chap 3, Prop 4.7(A) and Thm 4.10]{singhyper}. 

Recall the Picard-Lefschetz transformation \cite[Chap 3, 3.11]{singhyper}:

Suppose that $T_k$ is an endomorphism of $\Ho_4(\Mil_{f,x_i})$ as follows:
\[T_k(z)=y\pm <z,\Delta_k^i>_i\Delta_k^i\]
where $k=1,2$ and $\{\Delta_k^i\}$ is a distinguished basis of vanishing cycles $\Ho_4(\Mil_{f,x_i})$, see \cite[Chapter 3, Def 3.4]{singhyper}. The monodromy operator $r_i$ on $\Ho_4(\Mil_{f,x_i})$ is the composition of $T_1$ and $T_2$, see \cite[Chapter 3, 3.10]{singhyper}. Therefore, we have
\[r_j(y_j)=y_j\pm <y_j,\Delta_1^j>_j\Delta_1^j\pm(<y_j,\Delta_2^j>_j\pm <y_j,\Delta_1^j>_j<\Delta_1^j,\Delta_2^j>_j)\Delta_2^j.\]
If $y\neq 0$, then $y_j\neq 0$ for some $j$. Since $r_j(y_i)=y_i$ ($i\neq j$), we have
\[r_j(y)=(\sum\limits_{i\neq j}^my_i)+r_j(y_j)\]
\[=y\pm <y_j,\Delta_1^j>_j\Delta_1^j\pm(<y_j,\Delta_2^j>_j\pm <y_j,\Delta_1^j>_j<\Delta_1^j,\Delta_2^j>_j)\Delta_2^j.\]

Since the intersection form $<,>_j$ is nondegenerate, it follows that
\[r_j(y)=y+A\Delta_1^j+B\Delta_2^j=y+A'\delta_w^j+B'\delta_{w^2}^j\]
where $A$ or $B\neq 0$, $A'$ or $B'\neq 0$. It implies that
\begin{itemize}
\item $c:=A'\delta_w^j+B'\delta_{w^2}^j=r_j(y)-y\in W$ and
\item $r_j(c)=A'w\delta_w^j+B'w^2\delta_{w^2}^j \in W$.
\end{itemize}
In particular, we conclude that $\delta_w^j$ or $\delta_{w^2}^j$ is in the subrepresentation $W$. 
 It follows from Lemma \ref{conjugate} that $W$ contains $\Ho^4_w$ or $\Ho^4_{w^2}$.
\end{proof}

\begin{prop}\label{pic}
Suppose that $X_3$ is a smooth cubic fourfold defined by a cubic equation \[X_0^3+G(X_1,\ldots,X_5)=0.\] If $X_3$ is very general, then the middle $\Picard$ number of $X_3$ is one.
\end{prop}
\begin{proof}
Recall that Hodge loci are unions of countable many proper analytic subsets, see \cite{hodgelocus}. To prove the proposition, it suffices to show that there exists a smooth cubic fourfold defined by \[X_0^3+G(X_1,\ldots,X_5)=0\] whose middle Picard number is one. Consider the family $X_{\ao}\rightarrow \ao$ of cubic fourfolds which we describe at the beginning of this section. With the notations as before, if $b\in U$ is a very general point which is out of the Hodge locus of $X_{\ao}\rightarrow \ao$, then $\pi_1(U)$ preserves the rational Hodge classes $\mathrm{Hdg}(X_b)_{\Q}$, see \cite{hodgelocus}. Therefore, the vector space \[\mathrm{Hdg}(X_b)_{\C}\cap \Ho^4_{prim}(X_b,\C)\] is a subrepresentation of the $\pi_1(U)$-representation $\Ho^4_{prim}(X_b)$. It follows from Proposition \ref{irr} that $\mathrm{Hdg}(X_b)_{\C}\cap \Ho^4_{prim}(X_b,\C)$ is either zero, or $\Ho^4_{prim}(X_b)$, or $\Ho^4_w$, or $\Ho^4_{w^2}$. However, the last three cases are impossible since they contain either $\Ho^{3,1}$ or $\Ho^{1,3}$. So $\mathrm{Hdg}(X_b)_{\C}\cap \Ho^4_{prim}(X_b,\C)=0$. In other words, the middle Picard number of $X_b$ is one.

\end{proof}

\section{Lattices and Automorphisms}
In this section, we will use the theory of lattices to classify the automorphism groups of cubic fourfolds. For the theory of lattices, we refer to \cite{Nik} and \cite[Chapter 14]{K3}. We use the notations following \cite[Chapter 14]{K3}. Suppose that $F$ is a cubic form defining a smooth projective cubic fourfold $X_3$. 
Denote by $\sigma$ the $(3,1)$-form \[\mathrm{Res_{X_3}}\left(\frac{\sum\limits_{i=0}^5 (-1)^{i} X_i~dX_0 \wedge \cdots \widehat{dX_i} \cdots \wedge dX_5}{F^2}\right)\in \Ho^{3,1}(X_3).\] 
It follows from the Griffiths residues that \begin{equation} \label{residue}
\Ho^{3,1}(X_3)=\C< \sigma >.
\end{equation}
Let $T(X_3)$ be the smallest primitive integral sub-Hodge structure of $\Ho^4(X_3,\Z)$ containing $\Ho^{3,1}(X_3)$. It is the transcendental part of the Hodge structure on $\Ho^4(X_3,\Z)$. We have a natural action of the automorphism group $\Aut(X_3)$ on the one-dimensional vector space $\Ho^{3,1}(X_3)$. Since the automorphism group $\Aut(X_3)$ is finite, we have the following exact sequence
\begin{equation}\label{eq5}
\xymatrix{1 \ar[r]& \Aut_s(X_3) \ar[r] &\Aut(X_3) \ar[r]^p & \mu_m \ar[r] &1}
\end{equation}
where $\mu_m$ is the group of m-th roots of unity acting as multiplication on $\Ho^{3,1}(X_3)=\C$ and the kernel $ \Aut_s(X_3)$ is the subgroup of automorphisms preserving $\sigma$. \\
It follows from Poincar\'e duality and Lemma \ref{torsionfree} that $(\Ho^4(X_3,\Z),<,>)$ is a unimodular lattice. Denote this lattice by $\Lambda$. 
We define the middle Picard group $\MP(X_3)$ as the sublattice of $\Ho^4(X_3,\Z)$ which is generated by algebraic cycles
\[\MP(X_3):=\mathrm{Im}(cl:\mathrm{CH}^2(X_3)\rightarrow \Ho^4(X_3,\Z))=\Ho^4(X_3,\Z)\cap \Ho^{2,2}(X),\]
where the second equality follows from that the integral Hodge conjecture holds for smooth cubic fourfolds.

\begin{lemm}\label{tt}
If $g\in \Aut_s(X_3)$, then $g^*=\Id$ on $T(X_3)$.
\end{lemm}
\begin{proof}
Consider a map as follows
\[\psi:T(X_3)\rightarrow \C.\]
It associates to $x\in T(X_3)$ the intersection product $<x,\sigma>$ of $x$ and $\sigma$ . Since $g^*$ preserves the intersection product and $g \in \Aut_s(X_3)$, we have
\[\psi(g^*(x))=<g^*(x),\sigma>=<g^*(x),g^*\sigma>=\psi(x).\]
It follows that $g^*(x)-x\in T(X_3)\cap \MP(X_3)$.
Since we know $<\sigma, \overline{\sigma}> \neq 0$, we conclude $\Ker(\psi) \subseteq T(X_3)\cap \MP(X_3)$. 

We claim that $T^{\perp}(X_3)=\MP(X_3)$. In fact, since $\Ho^{3,1}(X_3)$ is orthogonal to $\MP(X_3)$, we conclude that $T(X_3) \subseteq \MP(X_3)^{\perp}$ by the minimality of $T(X_3)$. Therefore, we have $\MP(X_3)^{\perp \perp}\subseteq T(X_3)^{\perp}$. On the other hand, an integral class $x$ which is orthogonal to $T(X_3)$ is orthogonal to $\Ho^{3,1}(X_3)$. Therefore, we have $x\in \Ho^{2,2}(X_3).$ It implies that $T(X_3)^{\perp}\subseteq \MP(X_3)$. Obviously, we have $\MP(X_3)\subseteq \MP(X_3)^{\perp \perp} $. In summary, we get the following inclusions
\[\MP(X_3)\subseteq \MP(X_3)^{\perp\perp}\subseteq T(X_3)^{\perp}\subseteq \MP(X_3).\]It follows that $\MP(X_3)= T(X_3)^{\perp}$. We show the claim.

By the claim, if $y\in T(X_3)\cap \MP(X_3)$, then $<y,y>=0$ and $<y,\Ho^2>=0$ where $H$ is $c_1(\OO_{X_3}(1))$. It follows that $y=0$ by the Hodge Index theorem for $\Ho^4(X_3)$. Hence, we have $g^*(x)-x=0$ for all $x\in T(X_3)$. We prove the lemma.
\end{proof}
\begin{remark}\label{rmk}
 A similar argument shows that \[T(X_3)_{\Q}\cap \MP(X_3)_{\Q}=0.\]
\end{remark}
\begin{lemm}\label{lattice}
Suppose that $(\Lambda,<,>)$ is a unimodular lattice. Let $\Lambda_1$ be a primitive sublattice of $\Lambda$. Denote by $\Lambda_2$ the sublattice $\Lambda_1^{\perp}$. Define $A_{\Lambda_i}$ to be \[\frac{\Lambda_i^*}{\Lambda_i}\text{~where the inclusion~} \Lambda_i\subseteq \Lambda_i^* \text{~is induced by the intersection form $<,>$,}\]see \cite[Chapter 14]{K3}. If $\Lambda_1 \cap \Lambda_2=0$, then we have $A_{\Lambda_1} \simeq A_{\Lambda_2}$ via a natural identification.
\end{lemm}
\begin{proof}
We mimic the proof of \cite[Proposition 0.2, Chapter 14]{K3}. Suppose that $\psi$ is the composition of $i$ and $\pi$
\[\xymatrix{\psi:\La \ar@{^{(}->}[r]^i & \La_1^*\ar@{->>}[r]^{\pi} & \frac{\La^*_1}{\La_1}}\]
where $i$ sends $x$ to the linear map $<x,->$ on $\La_1$ and $\pi$ is the natural quotient. The kernel $\Ker(\psi)$ of $\psi$ is given by
\begin{center}
$\{x\in \La|$ there is $x'\in \La_1$ such that $<x,y>=<x',y>$ for all $y\in \La_1\}.$\\
\end{center}
Suppose that $x\in \Ker(\psi)$. There is $x'\in\La_1$ such that $x-x'\in \La_2$. It follows that \[x \in \La_1\oplus \La_2.\] Therefore, we have a natural inclusion
\[\overline{\psi}: \frac{\La}{\La_1 \oplus \La_2} \hookrightarrow \frac{\La_1^*}{\La_1}=A_{\La_1}.\] Since $\La_1 \hookrightarrow \La$ is primitive and $\La$ is unimodular, the map $\La\simeq \La^*\rightarrow \La^*_1$ is surjective. Therefore, the natural map $\overline{\psi}$ is surjective.  We conclude that $\overline{\psi}$ is an isomorphism. A similar argument shows the same result for $A_{\La_2}$. So we have a natural identification
\[A_{\La_1}\simeq \frac{\La}{\La_1\oplus \La_2} \simeq A_{\La_2}.\]
\end{proof}

\begin{lemm} \label{st}
$\Aut_s(X_3)=\{\Id\}$ if the rank of $\MP(X_3)$ is at most $2$.
\end{lemm}
\begin {proof}
By Proposition \ref{injective}, it suffices to show that $g^*=\Id$ on $\Ho^4(X_3,\Q)$ for all $g\in \Aut_s(X_3)$ if  the rank of $\MP(X_3)$ is at most $2$. By Lemma \ref{tt} and the fact that $\MP(X_3)+T(X_3)$ has the same rank as $\Ho^4(X,\Z)$, we just need to prove $g^*=\Id$ on $\MP(X_3)$ since $g^*=\Id$ on $T(X_3)$ by definition. We consider the following two cases.\\
\begin{enumerate}
\item If the rank of $\MP(X_3)$ is one, then $\MP(X_3)$ is generated by $H\cup H$ where $H=c_1(\OO_{X_3}(1))$. We have $g^*(\Ho^2)=\Ho^2$ since $g$ is linear by Lemma \ref{autlin}. It follows that $g^*=\Id$ on $\MP(X_3)$.
\item If the rank of  $\MP(X_3)$ is two, then we can assume that $\MP(X_3)=<\Ho^2,S>$ by the fact that the subgroup $<\Ho^2>$ generated by $\Ho^2$ is primitive. 
As in (i), we have $g^*(\Ho^2)=\Ho^2$. Suppose that $g^*(S)=a\Ho^2+bS$ where $a$ and $b$ are integers. It follows from $\mathrm{det} (g^*)= 1$ or $-1$ that $b=1$ or $-1$ (respectively). On the other hand, we have that \begin{equation}\label{eq6}
\Ho^2\cdot S=g^*(\Ho^2)\cdot g^*(S)=3a+b \Ho^2\cdot S.
\end{equation} If $\mathrm{det}(g^*)=1$, then $b=1$. It follows from (\ref{eq6}) that $a=0$. Therefore, we have $g^*(S)=S$. We prove the lemma. 

Suppose that $\mathrm{det}(g)=-1$. Then $b=-1$. By the equation (\ref{eq6}), we have \[a=\frac{2S\cdot \Ho^2}{3}, \mathrm{ ~i.e.,~~~} g^*(S)=\frac{2}{3}(S\cdot \Ho^2) \Ho^2-S.\] In particular, it follows that $S\cdot \Ho^2$ is divided by $3$. Therefore, the lattice $(\MP(X_3),<,>)$ is equivalent to the lattice
\begin{center}
$(\Ho^2)\oplus_{\perp}(S')=(3)\oplus_{\perp}(2n)$
\\with $<\Ho^2, \Ho^2>=3$, $<\Ho^2,S'>=0$
\end{center}
where $S'=S-\frac{S\cdot \Ho^2}{3}\Ho^2$. Note that the discrimiant of $\MP(X_3)$ is $6n$. It follows from \cite[Theorem 1.0.1]{Brend} that $n \geq 2$. 

It is clear that \[g^*(S')=-S'.\]  It follows that $g^*=(\Id,-\Id)$ on $ A_{\MP(X_3)}\cong \Z/3\Z\oplus \Z/2n\Z$. Since $n \geq 2$, we know $g^*\neq \Id$. On the other hand, by Lemma \ref{tt}, we know $g^*=\Id$ on $A_{T(X_3)}$. By the claim in the proof of Lemma \ref{tt}, we have $\MP(X_3)=T(X_3)^{\perp}$. Under the natural identification $A_{\MP(X_3)}=A_{T(X_3)}$ (see Lemma \ref{lattice}), we have $\Id=g^*\ne \Id$ which is absurd. It is a contradiction. So $\mathrm{det}(g)$ can not be $-1$.
\end{enumerate}
\end{proof}

\begin{cor} \label{cora}
The automorphism group $\Aut(X_3)$ is the cyclic group $\mu_m$ if the middle $\Picard$ number of $X_3$ is at most 2.
\end{cor}
\begin{proof}
It follows from the previous lemma and the exact sequence (\ref{eq5}).
\end{proof}

The rest of this section is to figure out what the number $m$ is.
\begin{lemm} \label{absurd}
Let $\La$ be a nondegenerate lattice of rank at most $2$. For $A_{\La}$, if one of the following assumptions holds
\begin{enumerate}
\item $\La=\Z<H>$ and $<H,H>$ is at least $3$,
\item $rk(\La)=2$ and $\disc(\La)>4$,
\end{enumerate}
then $\Id_{A_{\La}} \neq -\Id_{A_{\La}}$.
\end{lemm}

\begin{proof}
Assume (i). We have \[A_{\La}=\frac{\La^*}{\La}=\Z/(H\cdot H)\Z \simeq \Z/n\Z\] where $n\geq 3$. It follows the lemma.

Assume (ii). We identify $\La^*$ with $\{x\in \La_{\Q}|<x,\La> \subseteq \Z\}$. So we can pick up a basis $\{H,S\}$ of $\Lambda$ so that $\La^*=\Z<\frac{H}{N_1},\frac{S}{N_2}>$ for some integers $N_1$ and $N_2$. Therefore, we have
\[A_{\La}=\frac{\La^*}{\La}=\Z/N_1\Z \oplus \Z/N_2\Z.\] If $\Id_{A_{\La}} = -\Id_{A_{\La}}$, then $N_1,N_2\in \{1,2\}$. It contradicts with the hypothesis $\disc(\La)$($=|\frac{\La^*}{\La}|=N_1\cdot N_2$) $>$ 4.
\end{proof}

\begin{lemm} \label{phi}
The number $\varphi(m)$ divides the rank of $T(X_3)$ where $\varphi(-)$ is Euler's function. Suppose that the middle $\Picard$ number of $X_3$ is at most 2.  Then
\begin{itemize}
\item $\varphi(m)=1,2,22 $ if the rank of $\MP(X_3)$ is $1$,
\item $\varphi(m)=1$ if the rank of $\MP(X_3)$ is $2$.
\end{itemize}
\end{lemm}
\begin{proof}
We mimic the proof of \cite[Chapter 15 Corollary 1.14]{K3}. Let $T$ be an integral Hodge structure of K3-type such that it does not contain any proper primitive sub Hodge structure $T'$ of K3-type. 
\begin{center}\label{fact}
\textbf{Fact:}  For $a,a'\in \mathrm{End_{H.S}}(T)$, if $a=a'$ on $T^{2,0}$, then $a=a'$.
\end{center}
Note that the action of $\Aut(X_3)$ on $T(X_3)_{\C}$ is induced by $p:\Aut(X_3)\rightarrow \mu_m$, see (\ref{eq5}), where $\mu_m$ acts on $T(X_3)_{\C}$ faithfully. Let $f\in \Aut(X_3)$ be the automorphism such that \[p(f)=\xi\]where $\xi$ is a primitive m-th root of unity. Therefore, the minimal polynomial $\Phi_m(x)$ of $\xi$ divides the characteristic polynormial $G(x)$ of the linear map $f^*$ on $T(X_3)_{\Q}$. In order to show that $\phi (m)| rk(T(X_3))$, we simply remark that all non-trivial irreducible subrepresentations of $\mu_m=<p(f)>$ on $T(X_3)_{\Q}$ are of rank $\phi(m)$. \\
Indeed, otherwise, for some $n<m$, there exists a non-zero $b\in T(X_3)$ with $(f^n)^*(b)=b$. Suppose that $b=b_{1,3}+b_{2,2}+b_{3,1}$. It follows from Remark \ref{rmk} that $b_{3,1}\neq 0$. By the fact we mention at the beginning of the proof, we conclude that \[f^n (b_{3,1})=(\xi)^n b_{3,1}\neq b_{3,1}.\]It is absurd. 
In other words, as a representation of $\mu_m$, we have \[T(X_3)\simeq \Z[\xi]^{\oplus r}\]with $r=rk(T(X_3))/\phi(m)$. It follows that \[\phi(m)| rk(T(X_3))=23-rank(\MP(X_3)).\]
So we conclude that
\begin{itemize}
\item $\phi(m)\in \{1,2,11,22\}$ if $rank(\MP(X_3))=1$ and
\item $\phi(m)\in \{1,3,7,21\}$ if $rank(\MP(X_3))=2$.
\end{itemize}
By the elementary facts that 
\begin{itemize}
\item there is no $m$ such that $\phi(m)=11$ and
\item $\phi(m)$ is even if $m>2$,
\end{itemize} we prove the lemma.

\end{proof}

\begin{theorem}
The automorphism group $\Aut(X_3)$ is
\begin{itemize}
 \item $\{\Id\}$ or $\Z/3\Z$ if the middle $\Picard$ number of $X_3$ is one.
\item $\{\Id\}$ if the middle $\Picard$ number of $X_3$ is two.
\end{itemize}
\end{theorem}

\begin{proof}
Suppose that $rk(\MP(X_3))=1$. It follows from Corollary \ref{cora} that \[\Aut(X_3)=\mu_m.\] Using Lemma \ref{phi} and solving the equations \[\phi(m)=1,2,22,\] we conclude $m\in\{1,2,3,23,46\}$. By \cite[Theorem 1.6]{autocubic}, we know $m\neq 23$ or $46$. We claim $|\Aut(X_3)|$ is odd. It implies that $m\neq 2$ or $46$.

In fact, we suppose that $g\in \A(X_3)$ is of order $2$. Recall the fact that the action of $g^*$ on $T(X_3)$ is determined by the action on $T^{3,1}(X_3)$ (see the fact in the proof of Lemma \ref{phi}). It follows from this fact that $g^*=-\Id$ on $T(X_3)$. Therefore, we have $g^*=-\Id$ on $A_{T(X_3)} \cong A_{\MP(X_3)}$ (Lemma \ref{lattice}). On the other hand, we have $g^*=\Id$ on $\MP(X_3)$ since $g$ is linear. It follows that $g^*=\Id$ on $A_{\MP(X_3)}\cong A_{T(X_3)}$. However, it is absurd by Lemma \ref{absurd}.

We conclude that $m=1$ or $3$. According to a theorem of Deligne \cite[Exp XIX]{SGA7}, we know that a very general cubic fourfold is of middle Picard number one. By \cite{MM}, the automorphism group of a general cubic fourfold is trivial. Therefore, we show that there is a cubic fourfold of middle Picard number one with trivial automorphism group. For $m=3$, it follows from Proposition \ref{pic} that there is a cubic fourfold $X$ of middle Picard number one with an automorphism of order 3 as follows:
\[[X_0:X_1:\ldots:X_5]\rightarrow [\xi_3 X_0:X_1:\ldots:X_5],\]
where $\xi_3$ is $\mathrm{exp}(2\pi i/3)$.

For $rk(\MP(X_3))=2$, it follows from Corollary \ref{cora} that $\A(X_3)=\mu_m$. By Lemma \ref{phi} and the equation $\phi(m)=1$, we conclude that \[m=1 \mathrm{~or~} 2.\] We claim that $m\neq 2$. In fact, if  $g\in \A(X_3)$ is of order $2$, then $g^*=-\Id$ on $T(X_3)$ (the argument is similar to above). It follows that $g^*=-\Id$ on $A_{T(X_3)}\cong A_{\MP(X_3)}$. From the proof of Lemma \ref{st}, we know that there are only two possible cases for $g^*$ as follows:
\begin{enumerate}
\item $g^*=\Id$ on $\MP(X_3)$ (i.e., det $g^*|_{\MP(X_3)}=1$),
\item $g^*=(\Id,-\Id)$ on $\MP(X_3)=(3)\oplus_{\perp}(2n)$ where $n\geq 2$.
\end{enumerate}

However, by Lemma \ref{lattice} and Lemma \ref{absurd}, both cases are absurd. Therefore, we show $m\neq 2$. In summary, we prove the theorem.
\end{proof}
\appendix


\bibliographystyle{alpha}

\begin{thebibliography}{SGA73b}

\bibitem[AK77]{Compo}
Allen~B. Altman and Steven~L. Kleiman.
\newblock Foundations of the theory of {F}ano schemes.
\newblock {\em Compositio Math.}, 34(1):3--47, 1977.

\bibitem[BD85]{Bea2}
Arnaud Beauville and Ron Donagi.
\newblock La vari\'et\'e des droites d'une hypersurface cubique de dimension
  {$4$}.
\newblock {\em C. R. Acad. Sci. Paris S\'er. I Math.}, 301(14):703--706, 1985.

\bibitem[Bea82]{Bea}
Arnaud Beauville.
\newblock Les singularit\'es du diviseur {$\Theta $} de la jacobienne
  interm\'ediaire de l'hypersurface cubique dans {${\bf P}^{4}$}.
\newblock In {\em Algebraic threefolds ({V}arenna, 1981)}, volume 947 of {\em
  Lecture Notes in Math.}, pages 190--208. Springer, Berlin-New York, 1982.

\bibitem[Ber74]{Ber}
Pierre Berthelot.
\newblock {\em Cohomologie cristalline des sch\'emas de caract\'eristique
  {$p>0$}}.
\newblock Lecture Notes in Mathematics, Vol. 407. Springer-Verlag, Berlin-New
  York, 1974.

\bibitem[Blo72]{Bloch}
Spencer Bloch.
\newblock Semi-regularity and de{R}ham cohomology.
\newblock {\em Invent. Math.}, 17:51--66, 1972.

\bibitem[BO78]{BO}
Pierre Berthelot and Arthur Ogus.
\newblock {\em Notes on crystalline cohomology}.
\newblock Princeton University Press, Princeton, N.J.; University of Tokyo
  Press, Tokyo, 1978.

\bibitem[BO83]{derham}
P.~Berthelot and A.~Ogus.
\newblock {$F$}-isocrystals and de {R}ham cohomology. {I}.
\newblock {\em Invent. Math.}, 72(2):159--199, 1983.

\bibitem[CCZ15]{PAN}
Xuanyu~Pan Ci~Chen and Dingxin Zhang.
\newblock Automorphism and cohomology {II}: Complete intersections.
\newblock {\em Preprint 2015.}


\bibitem[CG72]{CG}
C.~Herbert Clemens and Phillip~A. Griffiths.
\newblock The intermediate {J}acobian of the cubic threefold.
\newblock {\em Ann. of Math. (2)}, 95:281--356, 1972.

\bibitem[Cho49]{Chow}
Wei-Liang Chow.
\newblock On the geometry of algebraic homogeneous spaces.
\newblock {\em Ann. of Math. (2)}, 50:32--67, 1949.

\bibitem[Del77]{SGA45}
P.~Deligne.
\newblock {\em Cohomologie \'etale}.
\newblock Lecture Notes in Mathematics, Vol. 569. Springer-Verlag, Berlin-New
  York, 1977.
\newblock S{\'e}minaire de G{\'e}om{\'e}trie Alg{\'e}brique du Bois-Marie SGA
  4${1{\o}er 2}$, Avec la collaboration de J. F. Boutot, A. Grothendieck, L.
  Illusie et J. L. Verdier.

\bibitem[Del81]{delsur}
P.~Deligne.
\newblock Rel\`evement des surfaces {$K3$} en caract\'eristique nulle.
\newblock In {\em Algebraic surfaces ({O}rsay, 1976--78)}, volume 868 of {\em
  Lecture Notes in Math.}, pages 58--79. Springer, Berlin-New York, 1981.
\newblock Prepared for publication by Luc Illusie.

\bibitem[Dim92]{singhyper}
Alexandru Dimca.
\newblock {\em Singularities and topology of hypersurfaces}.
\newblock Universitext. Springer-Verlag, New York, 1992.

\bibitem[DM69]{DM}
P.~Deligne and D.~Mumford.
\newblock The irreducibility of the space of curves of given genus.
\newblock {\em Inst. Hautes \'Etudes Sci. Publ. Math.}, (36):75--109, 1969.

\bibitem[Dol13]{Dolg}
Igor~V. Dolgachev.
\newblock Numerical trivial automorphisms of {E}nriques surfaces in arbitrary
  characteristic.
\newblock In {\em Arithmetic and geometry of {K}3 surfaces and {C}alabi-{Y}au
  threefolds}, volume~67 of {\em Fields Inst. Commun.}, pages 267--283.
  Springer, New York, 2013.

\bibitem[GAL11]{autocubic}
V{\'{\i}}ctor Gonz{\'a}lez-Aguilera and Alvaro Liendo.
\newblock Automorphisms of prime order of smooth cubic {$n$}-folds.
\newblock {\em Arch. Math. (Basel)}, 97(1):25--37, 2011.

\bibitem[Har70]{Hart}
Robin Hartshorne.
\newblock {\em Ample subvarieties of algebraic varieties}.
\newblock Lecture Notes in Mathematics, Vol. 156. Springer-Verlag, Berlin-New
  York, 1970.
\newblock Notes written in collaboration with C. Musili.

\bibitem[Has00]{Brend}
Brendan Hassett.
\newblock Special cubic fourfolds.
\newblock {\em Compositio Math.}, 120(1):1--23, 2000.

\bibitem[HT13]{BT}
Brendan Hassett and Yuri Tschinkel.
\newblock Hodge theory and {L}agrangian planes on generalized {K}ummer
  fourfolds.
\newblock {\em Mosc. Math. J.}, 13(1):33--56, 189, 2013.

\bibitem[Hua01]{Bott}
I-Chiau Huang.
\newblock Cohomology of projective space seen by residual complex.
\newblock {\em Trans. Amer. Math. Soc.}, 353(8):3097--3114 (electronic), 2001.

\bibitem[Huy15]{K3}
Daniel Huybrechts.
\newblock Lectures on k3 surfaces.
\newblock 2015.

\bibitem[Ill75]{Ill1}
Luc Illusie.
\newblock Report on crystalline cohomology.
\newblock In {\em Algebraic geometry ({P}roc. {S}ympos. {P}ure {M}ath., {V}ol.
  29, {H}umboldt {S}tate {U}niv., {A}rcata, {C}alif., 1974)}, pages 459--478.
  Amer. Math. Soc., Providence, R.I., 1975.

\bibitem[Ill79]{Ill}
Luc Illusie.
\newblock Complexe de de\thinspace {R}ham-{W}itt et cohomologie cristalline.
\newblock {\em Ann. Sci. \'Ecole Norm. Sup. (4)}, 12(4):501--661, 1979.

\bibitem[KM74]{KM}
Nicholas~M. Katz and William Messing.
\newblock Some consequences of the {R}iemann hypothesis for varieties over
  finite fields.
\newblock {\em Invent. Math.}, 23:73--77, 1974.

\bibitem[Lev01]{Tate}
Norman Levin.
\newblock The {T}ate conjecture for cubic fourfolds over a finite field.
\newblock {\em Compositio Math.}, 127(1):1--21, 2001.

\bibitem[Mat89]{Mat}
Hideyuki Matsumura.
\newblock {\em Commutative ring theory}, volume~8 of {\em Cambridge Studies in
  Advanced Mathematics}.
\newblock Cambridge University Press, Cambridge, second edition, 1989.
\newblock Translated from the Japanese by M. Reid.

\bibitem[Mil68]{milnor}
John Milnor.
\newblock {\em Singular points of complex hypersurfaces}.
\newblock Annals of Mathematics Studies, No. 61. Princeton University Press,
  Princeton, N.J.; University of Tokyo Press, Tokyo, 1968.

\bibitem[Mil80]{Milne}
James~S. Milne.
\newblock {\em \'{E}tale cohomology}, volume~33 of {\em Princeton Mathematical
  Series}.
\newblock Princeton University Press, Princeton, N.J., 1980.

\bibitem[MM64]{MM}
Hideyuki Matsumura and Paul Monsky.
\newblock On the automorphisms of hypersurfaces.
\newblock {\em J. Math. Kyoto Univ.}, 3:347--361, 1963/1964.

\bibitem[MN84]{Mukai}
Shigeru Mukai and Yukihiko Namikawa.
\newblock Automorphisms of {E}nriques surfaces which act trivially on the
  cohomology groups.
\newblock {\em Invent. Math.}, 77(3):383--397, 1984.

\bibitem[Mum08]{mum}
David Mumford.
\newblock {\em Abelian varieties}, volume~5 of {\em Tata Institute of
  Fundamental Research Studies in Mathematics}.
\newblock Published for the Tata Institute of Fundamental Research, Bombay; by
  Hindustan Book Agency, New Delhi, 2008.
\newblock With appendices by C. P. Ramanujam and Yuri Manin, Corrected reprint
  of the second (1974) edition.

\bibitem[Nik79]{Nik}
V.~V. Nikulin.
\newblock Integer symmetric bilinear forms and some of their geometric
  applications.
\newblock {\em Izv. Akad. Nauk SSSR Ser. Mat.}, 43(1):111--177, 238, 1979.

\bibitem[Ogu78]{Trans}
A.~Ogus.
\newblock Griffiths transversality in crystalline cohomology.
\newblock {\em Ann. of Math. (2)}, 108(2):395--419, 1978.

\bibitem[Ogu79]{Ogus}
Arthur Ogus.
\newblock Supersingular {$K3$} crystals.
\newblock In {\em Journ\'ees de {G}\'eom\'etrie {A}lg\'ebrique de {R}ennes
  ({R}ennes, 1978), {V}ol. {II}}, volume~64 of {\em Ast\'erisque}, pages 3--86.
  Soc. Math. France, Paris, 1979.

\bibitem[Poo05]{PB}
Bjorn Poonen.
\newblock Varieties without extra automorphisms. {III}. {H}ypersurfaces.
\newblock {\em Finite Fields Appl.}, 11(2):230--268, 2005.

\bibitem[PS08]{MHS}
Chris A.~M. Peters and Joseph H.~M. Steenbrink.
\newblock {\em Mixed {H}odge structures}, volume~52 of {\em Ergebnisse der
  Mathematik und ihrer Grenzgebiete. 3. Folge. A Series of Modern Surveys in
  Mathematics [Results in Mathematics and Related Areas. 3rd Series. A Series
  of Modern Surveys in Mathematics]}.
\newblock Springer-Verlag, Berlin, 2008.

\bibitem[PAN16]{PANL}
Xuanyu~Pan.
\newblock p-adic Deformations of Graph Cycles.
\newblock {\em Preprint 2016.}

\bibitem[Riz06]{Riz}
Jordan Rizov.
\newblock Moduli stacks of polarized {$K3$} surfaces in mixed characteristic.
\newblock {\em Serdica Math. J.}, 32(2-3):131--178, 2006.

\bibitem[SGA73a]{SGA7}
{\em Groupes de monodromie en g\'eom\'etrie alg\'ebrique. {II}}.
\newblock Lecture Notes in Mathematics, Vol. 340. Springer-Verlag, Berlin-New
  York, 1973.
\newblock S{\'e}minaire de G{\'e}om{\'e}trie Alg{\'e}brique du Bois-Marie
  1967--1969 (SGA 7 II), Dirig{\'e} par P. Deligne et N. Katz.

\bibitem[SGA73b]{SGA4}
{\em Th\'eorie des topos et cohomologie \'etale des sch\'emas. {T}ome 3}.
\newblock Lecture Notes in Mathematics, Vol. 305. Springer-Verlag, Berlin-New
  York, 1973.
\newblock S{\'e}minaire de G{\'e}om{\'e}trie Alg{\'e}brique du Bois-Marie
  1963--1964 (SGA 4), Dirig{\'e} par M. Artin, A. Grothendieck et J. L.
  Verdier. Avec la collaboration de P. Deligne et B. Saint-Donat.

\bibitem[SGA77]{SGA5}
{\em Cohomologie {$l$}-adique et fonctions {$L$}}.
\newblock Lecture Notes in Mathematics, Vol. 589. Springer-Verlag, Berlin-New
  York, 1977.
\newblock S{\'e}minaire de G{\'e}ometrie Alg{\'e}brique du Bois-Marie
  1965--1966 (SGA 5), Edit{\'e} par Luc Illusie.

\bibitem[Voi02]{Voisinhodge}
Claire Voisin.
\newblock {\em Th\'eorie de {H}odge et g\'eom\'etrie alg\'ebrique complexe},
  volume~10 of {\em Cours Sp\'ecialis\'es [Specialized Courses]}.
\newblock Soci\'et\'e Math\'ematique de France, Paris, 2002.

\bibitem[Voi13]{hodgelocus}
Claire Voisin.
\newblock Hodge loci.
\newblock In {\em Handbook of moduli. {V}ol. {III}}, volume~26 of {\em Adv.
  Lect. Math. (ALM)}, pages 507--546. Int. Press, Somerville, MA, 2013.

\end{thebibliography}

\end{document}